\documentclass[10pt]{amsart}

\usepackage{epic}
\usepackage{eepic}
\usepackage{amsmath}
\usepackage{amscd}
\usepackage{amsfonts}
\usepackage{amssymb}
\usepackage{latexsym}
\usepackage{pifont}

\newcommand{\ncm}{\newcommand}

\ncm{\beq}{\begin{equation}}
\ncm{\eeq}{\end{equation}}

\newtheorem{thm}{Theorem}[subsection]
\newtheorem{pro}[thm]{Proposition}
\newtheorem{lem}[thm]{Lemma}
\newtheorem{lem&def}[thm]{Definition \& Lemma}
\newtheorem{cor}[thm]{Corollary}
\newtheorem{thm&def}[thm]{Theorem \& Definition}

\newtheorem{defi}[thm]{Definition}

\newtheorem{exa}[thm]{Example}

\numberwithin{equation}{section}

\ncm{\ob}{\operatorname{ob}}
\ncm{\Nat}{\operatorname{Nat}}
\def\Set{\mathsf{Set}}
\ncm{\simp}{\Delta}
\def\Ab{\mathsf{Ab}}

\def\Ring{\mathsf{Ring}}

\def\Cat{\mathsf{Cat}}
\ncm{\CAT}{\mathsf{CAT}}
\def\Mon{\mathsf{Mon}}
\ncm{\mon}{\mathrm{mon}}
\ncm{\CMon}{\mathsf{CMon}}

\ncm{\Mnd}{\mathsf{Mnd}}

\ncm{\Fun}{\mathsf{Fun}}

\def\M{\mathsf{M}}
\ncm{\Cgd}{\mathsf{Cgd}}
\ncm{\BGD}{\mathsf{Bgd}}
\ncm{\BgdMap}{\mathsf{BgdMap}}
\ncm{\Bim}{\mathbb{BIM}}
\ncm{\YD}{\mathcal{YD}}
\ncm{\Gal}{\mathsf{Gal}}
\ncm{\Fib}{\mathsf{Fib}}
\ncm{\DHR}{\mathsf{DHR}}
\ncm{\lZ}{\mathcal{W}_\ell}
\ncm{\rZ}{\mathcal{W}_r}
\ncm{\Zl}{\mathcal{Z}_\ell}
\ncm{\Zr}{\mathcal{Z}_r}
\ncm{\BCA}{\mathsf{BCA}}
\ncm{\Bslice}{(B\downarrow\Ring)}
\ncm{\MonFun}{\mathsf{MonFun}}
\ncm{\Adj}{\mathsf{Adj}}
\ncm{\El}{\mathsf{Elt}\,}
\ncm{\Imp}{\mathsf{Imp}}
\ncm{\Add}{\mathsf{Add}}

\def\A{\mathcal{A}}

\def\C{\mathcal{C}}

\ncm{\E}{\mathcal{E}}
\ncm{\F}{\mathcal{F}}
\ncm{\G}{\mathcal{G}}
\ncm{\K}{\mathcal{K}}

\ncm{\Po}{\mathbb{P}}
\ncm{\R}{\mathcal{R}}
\ncm{\V}{\mathcal{V}}
\ncm{\W}{\mathcal{W}}
\ncm{\X}{\mathcal{X}}
\ncm{\Z}{\mathcal{Z}}
\ncm{\U}{\mathcal{U}}

\ncm{\End}{\operatorname{End}}
\ncm{\Aut}{\operatorname{Aut}}
\ncm{\Hom}{\operatorname{Hom}}
\ncm{\BiEnd}{\operatorname{BiEnd}}

\ncm{\Coker}{\operatorname{Coker}}
\ncm{\Ima}{\mathrm{Im}\,}
\ncm{\Coim}{\operatorname{Coim}}
\ncm{\kernel}{\operatorname{ker}}
\ncm{\coker}{\operatorname{coker}}
\ncm{\im}{\operatorname{im}}
\ncm{\coim}{\operatorname{coim}}
\ncm{\colim}{\operatorname{colim}}
\ncm{\Coinv}{\operatorname{Coinv}}
\ncm{\id}{\operatorname{id}}

\ncm{\Diff}{\operatorname{Diff}}
\ncm{\Der}{\operatorname{Der}}
\ncm{\ev}{\operatorname{ev}}
\ncm{\coev}{\operatorname{coev}}
\ncm{\dom}{\operatorname{dom}}
\ncm{\codom}{\operatorname{codom}}

\newcommand{\ci}{\circ}
\ncm{\bu}{\bullet}

\def\o{\otimes}
\def\x{\times}
\ncm{\homci}{\,{\sst\square}\,}
\ncm{\hombu}{\,{\sst\blacksquare}\,}

\ncm{\amalgo}[1]{\underset{\scriptscriptstyle #1}{\o}}
\ncm{\am}[1]{\underset{\scriptscriptstyle #1}{\o}}
\ncm{\amo}[1]{\underset{\scriptstyle #1}{\o}}
\ncm{\ractB}{\underset{\scriptscriptstyle B}{\ract}}
\ncm{\ractT}{\underset{\scriptscriptstyle T}{\ract}}
\ncm{\mash}{\Pisymbol{psy}{35}}
\ncm{\mashed}[1]{\underset{\scriptscriptstyle #1}{\Pisymbol{psy}{35}}}
\ncm{\cross}[1]{\underset{\scriptstyle #1}{\rtimes}}

\ncm{\oA}{\amalgo{A}}
\ncm{\oB}{\amalgo{B}}
\ncm{\oC}{\amalgo{C}}
\ncm{\oR}{\amalgo{R}}
\ncm{\oT}{\amalgo{T}}
\ncm{\oL}{\amalgo{L}}
\ncm{\oS}{\amalgo{S}}
\ncm{\oN}{\amalgo{N}}
\ncm{\oH}{\amalgo{H}}
\ncm{\oQ}{\amalgo{Q}}
\ncm{\ex}[1]{\underset{\scriptscriptstyle #1}{\x}}
\def\bra{\langle}
\def\ket{\rangle}
\ncm{\rarr}[1]{\stackrel{#1}{\longrightarrow}}
\ncm{\larr}[1]{\stackrel{#1}{\longleftarrow}}
\ncm{\mapsot}{\leftarrow\!\!\!\raisebox{1pt}{$\scriptscriptstyle |$}}
\def\cop{\Delta}
\ncm{\nullT}{^{(0)}}

\ncm{\nullB}{_{(0)}}
\def\oneB{_{(1)}}

\ncm{\coa}[1]{^{\langle {#1}\rangle}}

\def\eps{\varepsilon}
\ncm{\op}{\mathrm{op}}
\ncm{\coop}{\mathrm{coop}}
\ncm{\rev}{\mathrm{rev}}
\ncm{\dual}{\mathrm{dual}}
\ncm{\co}{\mathrm{co}}
\ncm{\Fi}{\varphi}
\ncm{\OR}{\overrightarrow}
\ncm{\OL}{\overleftarrow}

\ncm{\asso}{\mathbf{a}}
\ncm{\luni}{\mathbf{l}}
\ncm{\runi}{\mathbf{r}}

\ncm{\Ad}{\operatorname{Ad}}
\ncm{\I}{\mathcal{I}}

\def\iso{\stackrel{\sim}{\rightarrow}}

\def\ract{\triangleleft}

\ncm{\under}{\mbox{\rm\_}\,}

\ncm{\ZZ}{\mathbb{Z}}
\ncm{\GG}{\mathbf{G}}
\ncm{\FF}{\mathbb{F}}

\ncm{\Oo}{\mathcal{O}}
\ncm{\Ha}{\mathcal{H}}
\ncm{\Ve}{\mathcal{V}}
\ncm{\Pee}{\mathcal{P}}

\ncm{\adjoint}{\dashv}
\ncm{\into}{\hookrightarrow}
\ncm{\fgp}{\mathrm{fgp}}
\ncm{\adj}{\dashv}
\ncm{\sst}{\scriptstyle}
\ncm{\ssst}{\scriptscriptstyle}

\ncm{\eqby}[1]{\stackrel{(\ref{#1})}{=}}
\ncm{\lef}{{\sst <}}
\ncm{\righ}{{\sst >}}

\begin{document} 
 
\title[]{On the field algebra construction}

\author[]{Korn\'el Szlach\'anyi
}

\address{Research Institute for Particle and Nuclear Physics of the Hungarian Academy of Science\\
H-1525 Budapest, P. O. Box 49, Hungary}

\begin{abstract}
A pure algebraic variant of John Roberts' field algebra construction is presented and applied
to bialgebroid Galois extensions and certain generalized fusion categories.

\end{abstract}
\thanks{E-mail: \texttt{szlach@rmki.kfki.hu}\\
Partially supported by the Hungarian Scientific Research Fund, OTKA K 68195\\
The Arabian Journal for Science and Engineering, Vol 33, No 2C, pp 459-482 (2008)}

\maketitle

\vskip 1cm

\section{Introduction}

The reconstruction problem in Algebraic Quantum Field Theory (AQFT)
solved by the abstract duality theorem
of S. Doplicher and J. E. Roberts contains, as an important ingredient, the construction of field algebras from observables.
Although this construction relates fiber functors to field algebras, therefore relies only on Tannaka duality, it is
all the more interesting in the quantum group(oid) world which lies beyond the range of the Doplicher-Roberts
Theorem.

Forgetting certain of its structure elements Roberts' field algebra construction
\cite{Ro, Ha-Mu} has the following pure algebraic analogue.
For a (associative, unital) ring $B$, playing the role of the \textit{observable algebra}, we fix a monoidal category $\C$ of ring endomorphisms of $B$. The task is to determine a ring homomorphism $\rho:B\to A$ for each Abelian group
valued monoidal functor $F:\C\to\Ab$, i.e., for each `fiber functor' in the weakest sense. The \textit{field algebra} $A$ should be `large' enough to contain for each $\alpha\in\ob\C$ nonzero elements $a$ generating $\alpha$
by means of
$$
a\rho(b)=\rho(\alpha(b))a, \qquad\forall b\in B,
$$
and should be `small' enough to possess a universal property, in a sense which is also to be clarified.

In this paper we would like to point out to that the field algebra construction $F\mapsto\rho$ is the left
adjoint of another familiar construction in AQFT,
in which a fiber functor is constructed from the `family of Hilbert spaces'  within the field algebra \cite {Do-Ro}.
Moreover, we shall see that the field algebra $A$ can be obtained as the tensor product
$\Upsilon\amo{\C}F$ of the fiber functor with a fixed contravariant monoidal functor $\Upsilon$. The presheaf
$\Upsilon$ is solely determined by the data $\bra B,\C\ket$ and its existence characterizes the categories of interest.
The notion of tensor product `over $\C$' played important role in the Tannaka duality of Joyal and Street \cite{Jo-St}
but in a general context it appeared already in the book \cite{Fr}.
 
For certain categories $\C$ we can find fiber functors in a stronger sense, these are \textit{essentially
strong monoidal} \cite{Sz: Brussels} functors $F$ with the canonical $R$-$R$-bimodule structure of $F(\alpha)$ being finite projective from the right. For such functors the field algebra construction yields field algebras $A=\Upsilon\amo{\C}F$ that are right $H$-comodule algebras over the $R$-bialgebroids $H=F^\lef\amo{\C}F$. The bialgebroid appearing here
is the same structure that is associated to $F$ by Tannaka duality \cite{Phung}. The $R$-dual of $H$ as a coring is the
ring $\Nat(F,F)$ of endo-natural transformations of the `long' forgetful functor $F:\C\to\,_R\M_R\to\Ab$.

We discuss two special, finitary cases of the construction. In the first, the extension $\rho:B\to A$ will be
right adjoint in the 2-category $\Ring$ (see below). We show that this is the same thing as a Kleisli construction in $\Ring$. In the second, the category $\C$ is generated by direct summand diagrams from a single right adjoint endomorphism. This latter case `generalizes' the concept of fusion category in the non-semisimple direction
although the category is restricted to be a category of endomorphisms of a ring.
In both cases the field algebra is a right $H$-Galois extension of the observable algebra $B$ where the
symmetry object $H=F^\lef\amo{\C}F$ is a right $R$-bialgebroid that is finite projective as left $R$-module.

In the rest of this Introduction we summarize some prevalent concepts we use throughout the paper.

\subsection{The 2-category of rings}
Let $\Ring$ be the 2-category in which the objects are the small rings and the hom-categories $\Ring(A,B)$ are defined as follows. The objects in $\Ring(A,B)$ are the ring homomorphisms
(called morphisms for short) from $A$ to $B$ and the hom sets $\Hom(\alpha,\beta)$, for each parallel pair of morphisms $A\to B$, consist of those $t\in B$ which satisfy the intertwiner relation: $t\alpha(a)=\beta(a)t$ for all $a\in A$. 
Since the (vertical) composition in the hom-category $\Ring(A,B)$ is just multiplication in $B$, it
is tempting to denote the composite arrow $\alpha\rarr{t}\beta\rarr{s}\gamma$ simply by the
element $st\in B$. In this way the notation will somewhat hide the categorical structure but the
benefit is brevity.  
Horizontal composition for morphisms $A\rarr{\sigma}B\rarr{\rho}C$ is just composition of homomorphisms, denoted by juxtaposition, $\rho\o\sigma:=\rho\sigma$. 
Horizontal composition for intertwiners
$r:\rho_1\to\rho_2:B\to C$ and $s:\sigma_1\to\sigma_2:A\to B$
is defined by $r\o s:=r\rho_1(s)\equiv \rho_2(s)r:\rho_1\sigma_1\to\rho_2\sigma_2:A\to C$.

This 2-category embeds into the 2-category $\Ab$-$\Cat$ of small $\Ab$-categories as the full sub-2-category generated by the objects that are 1-object categories. Indeed, a (small) ring is nothing but a 1-object category enriched over $\Ab$ and a ring homomorphism is just and additive functor between such categories. Natural transformations $\varphi\to\psi:A\to B$ have therefore a single component which, being an arrow of $B$, is just an
element $t\in B$ for which the naturality condition takes the form of the intertwiner relation.

For any ring $B$ the endo-category $\Ring(B,B)$, the objects of which are the ring endomorphisms of $B$, is
a strict monoidal category with unit object $\iota_B:=\id_B$. Full monoidal subcategories $\C$ of this endo-category
will be the subject of interest for this paper.

As in any 2-category, one says that the 1-cell $\lambda:A\to B$ is left adjoint to the 1-cell $\rho:B\to A$ if
there exist 2-cells $e:\lambda\rho\to B$, $m:A\to\rho\lambda$ such that $\rho(e)m=1_A$ and $e\lambda(m)=1_B$.
This situation is denoted by $\lambda\dashv\rho$. 

For parallel adjunctions $m_i,e_i:\lambda_i\dashv\rho_i:B\to A$, $i=1,2$ we denote by $(\ )^\lef$ the isomorphism
\[
\Hom(\rho_2,\rho_1)\iso\Hom(\lambda_1,\lambda_2), \quad t\mapsto
t^\lef:=e_1\lambda_1(t)\lambda_1(m_2)
\]
and by $(\ )^\righ$ its inverse. If the 1-cells in question all have left adjoints then $(\ )^\lef$ extends to a contravariant and antimonoidal fully faithful endofunctor.

\subsection{The preorder 2-category} \label{ss: preorder}

For parallel 1-cells $\rho,\sigma:A\to B$ in $\Ring$ we define the relation
\[
\rho\leq\sigma\quad\text{if }\ \exists\ \rho\rarr{b_i}\sigma\rarr{a_i}\rho\ \text{such that}\ \sum^n a_i b_i=1_B\,.
\]
and call the diagram $\rho\rarr{b_i}\sigma\rarr{a_i}\rho$ a direct summand diagram.

If direct sums exist in $\Ring(A,B)$ then $\rho\leq\sigma$ iff $\rho$ is a direct summand of a finite direct sum of copies of $\sigma$.

Composing direct summand diagrams we see that for three parallel 1-cells $\rho\leq\sigma$ and $\sigma\leq\tau$ imply $\sigma\leq\tau$. Also, the trivial direct summand diagram yields $\rho\leq\rho$ for all 1-cells.
Therefore $\leq$ is reflexive and transitive.
We denote by $\Po(A,B)$ the preorder obtained in this way on the object set of the category $\Ring(A,B)$.
For 1-cells $\rho_1,\rho_2:B\to C$ and $\sigma_1,\sigma_2:A\to B$ we have the implication
$$
\rho_1\leq\rho_2\ \text{and}\ \sigma_1\leq\sigma_2\quad\Rightarrow\quad \rho_1\sigma_1\leq \rho_2\sigma_2
$$
which can be seen by tensoring the two direct summand diagrams. This implication defines the horizontal product
of a 2-category  $\Po$ in which the hom-categories are the preorders $\Po(A,B)$.

In particular, for each object $A$ the preorder $\Po(A,A)$ is a monoidal preorder. A monoid in this preorder
is a 1-cell $\mu:A\to A$ such that $\mu\mu\leq\mu$ and $\iota_A\leq\mu$.

The preordering is compatible with adjunctions in the following sense. If $\lambda_1\dashv\rho_1$ and $\lambda_2\dashv\rho_2$ then $\lambda_1\leq\lambda_2$ iff $\rho_1\leq\rho_2$. This follows by applying adjunction $(\ )^\lef$ or $(\ )^\righ$ to the direct summand diagrams.

The equivalence relation induced by the preordering $\leq$ is denoted by $\sim$, i.e., $\rho\sim\sigma$ if $\rho\leq\sigma$ and $\sigma\leq\rho$. The following result belongs to standard Morita theory and therefore
stated without proof.
\begin{lem} \label{Morita}
Let $\rho$ and $\sigma$ be parallel 1-cells in $\Ring$. Then vertical composition induces bimodule structures
over the rings $R=\End\rho$ and $S=\End\sigma$ on the following hom-groups: $_SU_R=\Hom(\rho,\sigma)$,
$_RV_S=\Hom(\sigma,\rho)$. If $\rho\leq\sigma$ then
\begin{enumerate}
\item $V_S$ and $_SU$ are finitely generated projective.
\item $_RV$ and $U_R$ are generators.
\item $_SU_R\cong\Hom_{-S}(V,S)$ and $_RV_S\cong\Hom_{S-}(U,S)$.
\item Left and right actions of $R$, respectively, define ring isomorphisms $R\to\End_{-S}(V)$ and $R\to\End_{S-}(U)$.
\item If in addition $\sigma\leq \rho$ then $R$ and $S$ are Morita equivalent rings.
\end{enumerate}
\end{lem}

\subsection{Why rings?}

The question may arise why we consider rings instead of $k$-algebras over a field or commutative ring $k$.
That is to say, why do we restrict the theory to such an `extreme' base ring as $\ZZ$? In fact we do not.
More precisely, there is a point of view from which the contrary is true.
Replacing $\Ab$ with some $\M_k$ and, accordingly, all $\Ab$-categories and $\Ab$-functors with $k$-linear ones,
would be the restriction. This point of view is that of AQFT according to which everything relevant should be
reconstructable from the observable data $\bra B, \C\ket$. If for example $B$ has a $K$-algebra structure and we
consider for $\C$ a full subcategory of $\Ring(B,B)$ which contains as objects only $K$-linear endomorphisms then
$\C$ will be automatically $K$-linear as well as the functor $\Upsilon$. Therefore taking only $K$-linear fiber functors
$F$ the field algebra construction will lead us to $K$ algebra maps $\rho:B\to A$. Since our steps never
explicitly use $K$, forgetting it, we may work on the $\Ab$-level and leave it to the specific physical problem and probably to additional, as yet unknown, principles to decide which $k$ is the best ground ring. This attitude is taken when we consider essentially strong monoidal functors $F$ to $\Ab$ instead of strong monoidal ones to some bimodule category and determine the (noncommutative) ground ring $R$ as the image under $F$ of the unit monoid.

\section{Extensions of $B$ and fiber functors on $\C$}

Our starting point is a small monoidal $\Ab$-category $\C$ which is embedded as a full monoidal subcategory
of $\Ring(B,B)$ for some ring $B$. The characterizing property of such categories $\C$ is the existence of certain Abelian group valued presheaves $\Upsilon$ on $\C$ 
in terms of which the field algebra construction can be written as briefly as
$\Upsilon\amo{\C}\under$.

\subsection{The basic presheaf on $\C$}

\begin{pro} \label{pro: Ups}
For a monoidal $\Ab$-category $\C$ to be embedable, as a (full) monoidal subcategory, into some
endo-category of $\Ring$ it is necessary and sufficent to exist a faithful, additive, essentially strong monoidal functor
$\Upsilon:\C^\op\to\Ab$ such that in its canonical factorization
\[
\C^\op\to \,_B\M_B\to\,_B\M\to\Ab
\]
the $\C^\op\to\,_B\M$ part has constant object map (and the strong part $\C^\op\to\,_B\M_B$ is full).
Here $B$ denotes the image of the identity monoid under the functor.
\end{pro}
\begin{proof}
\textit{Necessity:} For a ring $B$ let $\C$ be a monoidal subcategory of $\Ring(B,B)$.
Define $\Upsilon:\C^\op\to \Ab$ as the functor with constant object map $\Upsilon(\alpha)=B$ $\forall\alpha\in\ob\C$
and for $t\in\C(\alpha,\beta)$ let $\Upsilon(t):B\to B$ be right multiplication by $t$. We set the following monoidal structure
\begin{align*}
\Upsilon_{\alpha,\beta}&:B\o B\to B\,,\quad b_1\o b_2\mapsto b_1\alpha(b_2)\\
\Upsilon_0&:\ZZ\to B\,,\qquad 1\mapsto 1_B\,.
\end{align*}
Clearly, $\Upsilon$ is a faithful additive functor.
As for the monoidality, the associativity and unitality constraints can be easily verified.
Naturality of $\Upsilon_{\bullet,\bullet}$ for $\alpha\rarr{t}\alpha'$ and
$\beta\rarr{s}\beta'$ follows from the identity $b_1\alpha'(b_2)t\alpha(s)=b_1t\alpha(b_2s)$. That this monoidal structure is essentially strong means that $\Upsilon_{\alpha,\beta}$ is a coequalizer in
\begin{equation*}
\Upsilon(\alpha)\o\Upsilon(\iota)\o\Upsilon(\beta)
\parbox{70pt}{
\begin{picture}(70,40)
\put(5,17){\vector(1,0){60}} \put(15,28){$1\o\Upsilon_{\iota,\beta}$}
\put(5,23){\vector(1,0){60}} \put(15,6){$\Upsilon_{\alpha,\iota}\o 1$}
\end{picture}
}
\Upsilon(\alpha)\o\Upsilon(\beta) 
\parbox{50pt}{
\begin{picture}(50,40)
\put(5,20){\vector(1,0){38}} \put(46,20){\vector(1,0){0}} \put(15,25){$\Upsilon_{\alpha,\beta}$}
\end{picture}
}
\Upsilon(\alpha\beta)\,.
\end{equation*}
In our present situation $\Upsilon(\iota)$ is the ring $B$ with multiplication $\Upsilon_{\iota,\iota}$ and for every 
object $\alpha$ the $\Upsilon(\alpha)$ is the underlying Abelian group of a $B$-$B$-bimodule $\hat\Upsilon(\alpha)$ with regular action $\Upsilon_{\iota,\alpha}$ from the
left and with $\alpha$-twisted regular action $\Upsilon_{\alpha,\iota}$ from the right. Therefore the above diagram
is nothing but the coequalizer defining the tensor product $B_\alpha\oB B_\beta\cong B_{\alpha\beta}$ of the 
appropriate $B$-$B$-bimodule structures on $B$.
Note that the $\bra \Upsilon(\alpha),\Upsilon_{\iota,\alpha}\ket$ being the left regular $B$-module for all $\alpha$ means that not only $\Upsilon$ has constant object map but its $\C^\op\to\,_B\M$ factor as well. (Since every bimodule map $\hat \Upsilon(\beta)\to \hat \Upsilon(\alpha)$ is right multiplication on $B$ with an intertwiner
$\alpha\to\beta$, if $\C\subset\Ring(B,B)$ is a full subcategory  then $\hat\Upsilon$ is full.)

\textit{Sufficiency:} Let $\C$ be any monoidal $\Ab$-category, not necessarily strict. We denote the coherence isomorphisms for associativity and unitality (of $\C$) by $\asso_{\alpha,\beta,\gamma}$ and $\luni_\alpha$, $\runi_\alpha$, respectively. For $U:\C^\op\to\Ab$ a functor of the required type we can construct the ring
$B:=\bra U(\iota),U(\luni^{-1}_\iota)\ci U_{\iota,\iota},U_0\ket$ and the $B$-$B$-bimodules
$\hat U(\alpha):=\bra U(\alpha),U(\luni^{-1}_\alpha)\ci U_{\iota,\alpha}, U(\runi^{-1}_\alpha)\ci U_{\alpha,\iota}\ket$.
Then by assumption each $\hat U(\alpha)$, as a left $B$-module, is the left regular module $_BB$. Therefore the
right $B$-action must be of the form $x\cdot b=x\tilde\alpha(b)$, $x\in U(\alpha)=B$, $b\in B$ for some ring endomorphism $\tilde\alpha$ of $B$. For arrows $t\in\C(\alpha,\beta)$ the $U(t)$ lifts to a $B$-$B$-bimodule map $\hat U(t):\hat U(\beta)\to\hat U(\alpha)$ which means that $U(t)$ is right multiplication by an element
$\tilde t\in B$ such that $y\tilde t \tilde\alpha(b)=y\tilde\beta(b)\tilde t$ for all $y\in U(\beta)=B$ and $b\in B$.
This proves that $\tilde t$ is an arrow $\tilde\alpha\to\tilde\beta$ in $\Ring(B,B)$ and the construction $\tilde\ $ provides a functor from $\C$ to $\Ring(B,B)$. This functor is faithful since $\tilde t=0$ implies $U(t)=0$ but $U$ is faithful. This functor is monoidal since $\hat U(\alpha)\oB\hat U(\beta)\cong\hat U(\alpha\o\beta)$ implies $\tilde\alpha\tilde\beta=\widetilde{\alpha\beta}$. (Since right multiplication with any intertwiner $\tilde\alpha\to\tilde\beta$ produces a bimodule map $\hat U(\beta)\to\hat U(\alpha)$, if $\hat U$ is full then
$\tilde\ $ is full.)
\end{proof}

\subsection{Cross products of $B$ with fiber functors} \label{ss: ff>rho}

We recall the $\Ab$-version of the definition \cite{Jo-St} of the tensor product of a contravariant and a covariant functor.
For a small $\Ab$-category $\C$ and a pair of additive functors $\Upsilon:\C^\op\to\Ab$ and $F:\C\to\Ab$ one  defines the Abelian group $\Upsilon\amalgo{\C}F$, called the tensor product of $\Upsilon$ and $F$, as the coequalizer
\begin{equation} \label{tensor}
\coprod_{\alpha,\beta\in\ob\C}\Upsilon(\beta)\o\C(\alpha,\beta)\o F(\alpha)
\parbox{50pt}{
\begin{picture}(50,40)
\put(5,17){\vector(1,0){40}} \put(22,27){$\mathbf{L}$}
\put(5,23){\vector(1,0){40}} \put(22,8){$\mathbf{R}$}
\end{picture}
}
\coprod_{\alpha\in\ob\C}\Upsilon(\alpha)\o F(\alpha)
\parbox{50pt}{
\begin{picture}(50,40)
\put(5,20){\vector(1,0){38}} \put(46,20){\vector(1,0){0}} \put(17,27){$c_{\Upsilon,F}$}
\end{picture}
}
\Upsilon\amalgo{\C} F
\end{equation}
in $\Ab$ where the maps $\mathbf{L}$, $\mathbf{R}$ are defined by 
\begin{align*}
\mathbf{L}\ci i_{\alpha,\beta}(y\o t\o x)&=i_\alpha(\Upsilon(t)y\o x)\\
\mathbf{R}\ci i_{\alpha,\beta}(y\o t\o x)&=i_\beta(y\o F(t)x)
\end{align*}
for $x\in F(\alpha)$, $y\in\Upsilon(\beta)$ and $t\in\C(\alpha,\beta)$.
Equivalently, $\Upsilon\amo{\C}F$ is the coend of the functor $\Upsilon \o F:\C^\op\x\C\to\Ab$.
Notice that for representable $\Upsilon$ the tensor product reduces to
\begin{equation} \label{rep tensor F}
\C(\under,\alpha)\amo{\C}F\cong F(\alpha)\,,\qquad \alpha\in\ob\C\,.
\end{equation}
Based on this observation one can show \cite{CWM} that the functor $\under\amo{\C}F$ is the left Kan extension
of $F$ along the Yoneda embedding $Y:\C\to\Add(\C^\op,\Ab)$.

In order to understand the sense in which $\amalgo{\C}$ is a tensor product let us derive the corresponding hom-tensor relation. Let $Z$ be an Abelian group and define the (covariant) functor $\Ab(\Upsilon,Z):\C\to\Ab$
the object map of which is $\alpha\mapsto\Ab(\Upsilon(\alpha),Z)$. If $F:\C\to\Ab$ is any $\Ab$-functor then a natural transformation $\nu:F\to\Ab(\Upsilon,Z)$ is a collection of group homomorphisms $\nu_\alpha$ satisfying
\[
\nu_\beta(F(t)x)(y)=\nu_\alpha(x)(\Upsilon(t)y)\,,\qquad y\in\Upsilon(\beta),\ t\in\C(\alpha,\beta),\ x\in F(\alpha)\,.
\]
In other words, using $\Ab(F(\alpha),\Ab(\Upsilon(\alpha),Z))\cong\Ab(\Upsilon(\alpha)\o F(\alpha),Z)$, we have a collection of maps $\tilde{\nu}_\alpha:\Upsilon(\alpha)\o F(\alpha)\to Z$ such that
\[
\tilde{\nu}_\beta(y\o F(t)x)=\tilde{\nu}_\alpha(\Upsilon(t)y\o x)\,,\qquad y\in\Upsilon(\beta),\ t\in\C(\alpha,\beta),\ x\in F(\alpha)\,.
\]
This means precisely that 
\[
\tilde{\mu}:=\coprod_{\alpha\in\ob\C}\tilde{\nu}_\alpha:\coprod_{\alpha\in\ob\C}\Upsilon(\alpha)\o F(\alpha)\to Z
\]
satisfies $\tilde{\mu}\ci\mathbf{L}=\tilde{\mu}\ci\mathbf{R}$. The unique $\mu:\Upsilon\amalgo{\C}F\to Z$ for which $\mu\ci c_{\Upsilon,F}=\tilde{\mu}$ yields therefore an isomorphism of Abelian groups
\begin{equation}
\Nat(F,\Ab(\Upsilon,Z))\iso\Ab(\Upsilon\amalgo{\C}F,Z)\,,\qquad\nu\mapsto \mu
\end{equation}
for each Abelian group $Z$. Naturality in $Z\in\Ab$ and in $F\in\Fun(\C,\Ab)$ can be easily verified. In this way we have constructed
a functor $\Upsilon\amalgo{\C}\under:\Fun(\C,\Ab)\to\Ab$ which has a right adjoint.

In a similar way one can show that $\under\amo{\C}F$ as a functor $\Add(\C^\op,\Ab)\to\Ab$ also has a right
adjoint, namely the functor $X\mapsto\Ab(F\under,X)$.

In case of $\C$ is a monoidal category and both $\Upsilon$ and $F$ are monoidal functors the tensor product $\Upsilon\amalgo{\C}F$
becomes a monoid in $\Ab$. As a matter of fact, we can define multiplication on rank-1 elements $y\amalgo{\C}x$ by
\begin{align}
(y_1\amalgo{\C}x_1)(y_2\amalgo{\C}x_2)
&:=\Upsilon_{\alpha,\beta}(y_1\o y_2)\amalgo{\C} F_{\alpha,\beta}(x_1\o x_2)\\
&\text{for }x_1\in F(\alpha),\ x_2\in F(\beta),\ y_1\in\Upsilon(\alpha),\ y_2\in\Upsilon(\beta). \notag
\end{align}
Let us check that it is well-defined: If $x_1=F(t)x'_1$ for some $t:\alpha'\to\alpha$ and $x'_1\in F(\alpha')$ and also
$x_2=F(s)x'_2$ for some $s:\beta'\to\beta$ and $x'_2\in F(\beta')$ then, introducing $y_1':=\Upsilon(t)y_1$, $y'_2:=\Upsilon(s)y_2$ we find that
\begin{align*}
\Upsilon_{\alpha,\beta}(y_1\o y_2)\amalgo{\C} F_{\alpha,\beta}(x_1\o x_2)&=
\Upsilon_{\alpha,\beta}(y_1\o y_2)\amalgo{\C}F(t\o s)F_{\alpha',\beta'}(x'_1\o x'_2)\\
&=\Upsilon(t\o s)\Upsilon_{\alpha,\beta}(y_1\o y_2)\amalgo{\C}F_{\alpha',\beta'}(x'_1\o x'_2)\\
&=\Upsilon_{\alpha',\beta'}(y'_1\o y'_2)\amalgo{\C} F_{\alpha',\beta'}(x'_1\o x'_2).
\end{align*}
Associativity follows from the associativity constraints on $\Upsilon_2$  and $F_2$. The unit element of $\Upsilon\amalgo{\C}F$ is
$\Upsilon_0(1)\amalgo{\C}F_0(1)$. Since the image of the unit object $\iota\in\C$ of any monoidal $\Ab$-functor $\C\to\Ab$ is a ring, the ring $\Upsilon\amalgo{\C}F$ is endowed with two ring homomorphisms
\begin{align}
\label{eq: B->F}
\rho_F\,:\,\Upsilon(\iota)&\to \Upsilon\amalgo{\C}F\,,\quad b\mapsto b\amalgo{\C}F_0(1)\\
\label{eq: R->F}
\pi_F\,:\,F(\iota)&\to \Upsilon\amalgo{\C}F\,,\quad r\mapsto \Upsilon_0(1)\amalgo{\C}r\,.
\end{align}

\begin{defi} \label{def: FA}
For a ring $B$ and a full monoidal subcategory $\C$ of $\Ring(B,B)$ let $\Upsilon$ be the basic presheaf of Proposition \ref{pro: Ups}. Then
for any additive monoidal functor $F:\C\to\Ab$ we define the field algebra $A$ as the
ring $\Upsilon\amalgo{\C}F$ and call the ring homomorphism $\rho_F:B\to A$ of (\ref{eq: B->F})
the field algebra extension of $B$ associated to $F$.
\end{defi}

For a field algebra the general coequalizer defining $\Upsilon\amalgo{\C}F$ specializes to
\begin{equation*} \label{tensor spec}
B\o\coprod_{\alpha,\beta\in\ob\C}\Hom(\alpha,\beta)\o F(\alpha)
\parbox{50pt}{
\begin{picture}(50,40)
\put(5,17){\vector(1,0){40}} \put(22,27){$\mathbf{L}$}
\put(5,23){\vector(1,0){40}} \put(22,8){$\mathbf{R}$}
\end{picture}
}
B\o\coprod_{\alpha\in\ob\C}F(\alpha)
\parbox{50pt}{
\begin{picture}(50,40)
\put(5,20){\vector(1,0){38}} \put(46,20){\vector(1,0){0}}
\end{picture}
}
A\,.
\end{equation*}
Therefore the elements of $A$ are finite sums of words $b\amalgo{\C}x$ with $b\in B$, $x\in F(\alpha)$, $\alpha\in\ob\C$ subject to the relations
\begin{equation} \label{tensor rels}
bt\am{\C} x=b\am{\C} F(t)x\quad \text{for }b\in B,\ t\in \C(\alpha,\beta),\ x\in F(\alpha).
\end{equation}
For elements of $\Upsilon\amo{\C}F$ 
we shall use the alternative and more informative notation $b\amo{\alpha}x$ for $b\amalgo{\C}x$ if $x$ belongs to $F(\alpha)$. Then multiplication on $A$ takes the form
\begin{equation} \label{multip}
(b\amo{\alpha}x)(b'\amo{\beta}x')=b\alpha(b')\amo{\alpha\beta}F_{\alpha,\beta}(x\o x')\,.
\end{equation}
The unit element is $1_A=1_B\amo{\iota}1_R$ where $1_R=F_0(1)$ is the unit of the ring $R:=F(\iota)$ with multiplication $F_{\iota,\iota}$. The field algebra $A$ is always an $R$-ring by the ring homomorphism
(\ref{eq: R->F}) such that $\pi_F(R)$ commutes with $\rho_F(B)$.

\subsection{When direct sums exist}

Roberts' original definition of the field algebra is formulated in terms of certain equivalence classes of triples
$\bra b,\alpha,x\ket$, $\alpha\in\ob\C$, $b\in B$, $x\in F(\alpha)$ and it uses the existence of direct sums in $\C$
to show that equivalence classes can be added. On the other hand, the coequalizer (\ref{tensor}) is automatically an Abelian group even if $\C$ has no direct sums. Nevertheless, it is not obvious why the two definitions coincide in the
case when $\C$ has direct sums. This Subsection is devoted clarifying this point.
Note that monoidality plays no role in the argument.

The category $\Ring(A,B)$ is never additive if $B\neq 0$ since it is lacking a zero object. But it can have
binary\footnote{We use `binary' instead of `finite' since the latter is usually meant to include the case of
a zero number of terms which, for direct sums, would correspond to the zero object.} direct sums and this depends only on the structure of the ring $B$. Namely,
$\Ring(A,B)$ has direct sums of any pair of objects iff it has a direct sum for a single pair of objects iff there exist elements 
$p_1,p_2,i_1,i_2\in B$ such that $p_1i_1=p_2i_2=1_B$, $p_1i_2=p_2i_1=0$ and $i_1p_1+i_2p_2=1_B$ if and only
if $B\oplus B\cong B$ as left (or right) $B$-modules. In particular, if $\Ring(B,B)$ has direct sums then $B$ is necessarily lacking Invariant Base Number.

Assume $_B(B\oplus B)\cong\,_BB$ and let $\C\subset\Ring(B,B)$ be a full subcategory closed under direct sums. Then
for any additive functor $F:\C\to \Ab$ the elements of the Abelian group $\Upsilon\amo{\C}F$ are all
rank 1 tensors. As a matter of fact,
\begin{equation} \label{tensor addition}
\begin{gathered} 
b_1\amo{\alpha}x_1 + b_2\amo{\beta}x_2 =  b_1p_1\amo{\sigma}F(i_1)x_1 +  b_2p_2\amo{\sigma}F(i_2)x_2=\\
=b_1p_1\amo{\sigma}F(i_1)x_1 +  b_2p_2\amo{\sigma}F(i_2)x_2+b_1p_1\amo{\sigma}F(i_2)x_2
+b_2p_2\amo{\sigma}F(i_1)x_1=\\
= (b_1p_1+b_2p_2)\amo{\sigma}(F(i_1)x_1 + F(i_2)x_2) 
\end{gathered}
\end{equation}
where $p_1,p_2,i_1,i_2$ are chosen as above and $\sigma(b):=i_1\alpha(b)p_1+i_2\beta(b)p_2$ is a direct sum of $\alpha$ and $\beta$.

This suggests that in the presence of direct sums the $\Upsilon\amo{\C}F$ could be computed using $\Set$-valued functors.
Consider the coequalizer (\ref{tensor}) in $\Set$, i.e., with coproducts replaced by disjoint unions and $\o$ by Cartesian product $\x$ of sets. Then one can recognize the tensor product as the colimit
\begin{equation} \label{set colim}
\Upsilon\ex{\C}F\ =\ \colim\left((\El \Upsilon)^\op\rarr{\phi}\C\rarr{F}\Ab\rarr{\U}\Set\right)
\end{equation}
where $\El\Upsilon$ is the category of elements of $\Upsilon$. This category has the pairs $\bra \alpha,b\ket$
as its objects, where $\alpha\in\ob\C$ and $b\in B$, and its arrows $\bra\alpha,b\ket\rarr{t}\bra\beta,b'\ket$ are
those arrows $t\in\C(\beta,\alpha)$ for which $bt=b'$. This category is equipped with the obvious forgetful functor $\phi^\op$ to $\C^\op=\dom\Upsilon$. The colimit itself is the set $\pi_0(\El \U F\phi)$ of connected components
in the category of elements of the composite $(\El\Upsilon)^\op\to\Set$. The objects in the latter category are the
triples
$\bra b,\alpha,x\ket$, $\alpha\in\ob\C$, $b\in B$, $x\in F(\alpha)$, and the arrows $\bra b,\alpha,x\ket\rarr{t}
\bra b',\beta,x'\ket$ are the $t\in\C(\alpha,\beta)$ such that $b=b't$ and $F(t)x=x'$. Two objects are called connected
when there is a zig-zag of arrows $\sst\nearrow\nwarrow\nearrow\nwarrow \nearrow\nwarrow$ from one to
the other. The colimiting cone from the functor $\U F\phi$ is therefore given by
\begin{equation} \label{cone}
\tau_{\bra\alpha,b\ket}:\U F\phi(\bra\alpha,b\ket)\to\pi_0(\El \U F\phi),\quad x\mapsto [b,\alpha,x]
\end{equation}
where $[b,\alpha,x]$ denotes the equivalence class, i.e., connected component of the object $\bra b,\alpha,x\ket$.
Our task is therefore to show that $\tau$ is underlying a colimiting cone in $\Ab$ provided $\C$ has binary direct
sums.

Let us assume that $\C$ has binary direct sums. If
$$
\alpha \overset{p}{\underset{i}{\leftrightarrows}} \sigma
\overset{p'}{\underset{i'}{\rightleftarrows}}\beta
$$
is a direct sum diagram in $\C$ then
\begin{equation} \label{product in ElUps}
\bra \alpha, b\ket \rarr{i}\bra\sigma,a\ket\larr{i'}\bra\beta,b'\ket
\end{equation}
with $a=bp+b'p'$, is a product diagram in $\El\Upsilon$. Therefore $\El\Upsilon$ has
binary products.
In the next Lemma $\I$ refers to the category $(\El\Upsilon)^\op$.

\begin{lem} \label{lem: U}
Let $\I$ be a small category with binary coproducts and let $X:\I\to\Ab$ be any functor. If $\tau_O:\U X(O)\to A$,
$O\in\ob\I$, is a colimiting cone in $\Set$ from $\U X$ then $\tau$ is the image under $\U$ of a colimiting cone in
$\Ab$ from $X$.
\end{lem}
\begin{proof}
The colimit of $\U X$ has the canonical presentation as a set of equivalence classes in a disjoint union
\begin{align*}
A&=\bigsqcup_{O\in\ob\I} \U X(O)
\begin{picture}(16,20)\put(0,-10){\line(8,20){8}}\put(8,-5){$\sim$}\end{picture}
\qquad \text{where}\\
&\bra O,x\ket \sim\bra O',x'\ket\ \text{means the existence of arrows in } \El \U X\\
&\bra O,x\ket =e_0\rarr{}e_1\larr{}e_2\rarr{}e_3\larr{}\dots\larr{}e_{2n}\ =\bra O',x'\ket
\end{align*}
with colimiting cone $\tau_O: \U X(O)\to A$, $x\mapsto [O,x]$ and with $[O,x]$ denoting the equivalence class of $\bra O,x\ket$. 

At first we define addition on the set $A$ by
\begin{equation} \label{addition}
[O_1,x_1]+[O_2,x_2]:=[O, X(u_1)x_1+X(u_2)x_2]
\end{equation}
where $O_1\rarr{u_1}O\larr{u_2}O_2$ is a coproduct diagram in $\I$. In order to see that this is well-defined let
$O'_1\rarr{v_1}O'\larr{v_2}O'_2$ be another coproduct diagram and let $t_i:O'_i\to O_i$, $x'_i\in X(O'_i)$,
$i=1,2$ be such that $X(t_i)x'_i=x_i$ for $i=1,2$.
Then the unique arrow $s$ in the diagram
$$
\begin{CD}
O'_1@>v_1>>O'@<v_2<<O'_2\\
@V{t_1}VV@VVsV@VV{t_2}V   \\             
O_1@>u_1>>O@<u_2<<O_2
\end{CD}
$$
yields an arrow $\bra O',X(v_1)x'_1+X(v_2)x'_2\ket\rarr{s}\bra O,X(u_1)x_1+X(u_2)x_2\ket$ in
the category $\El \U X$. Indeed, $X(s)(X(v_1)x'_1+X(v_2)x'_2)=X(u_1)x_1+X(u_2)x_2$.
This suffices to see that the addition (\ref{addition}) does not depend on the choice of the representant objects $\bra O_i,x_i\ket$. Associativity is obvious and the neutral element is $[O,0]$ for any object $O$.

The next step is to show that the colimiting cone $\tau$ can be lifted to $\Ab$, i.e., each component $\tau_O$ is
additive. This follows from that for a coproduct $O\rarr{u_1}S\larr{u_2}O$ we can take the codiagonal
$s:S\to O$ defined by $su_i=O$, $i=1,2$, hence
\begin{align*}
\tau_O(x_1)+\tau_O(x_2)&=[S,X(u_1)x_1+ X(u_2)x_2]=\\
&=[O,X(s)(X(u_1)x_1+ X(u_2)x_2)]=[O,x_1+x_2]=\\
&=\tau_O(x_1+x_2)\,,\qquad x_1,x_2\in X(O)\,.
\end{align*}
Therefore $\tau$ is a cone from $X$ to the Abelian group $A$.

Finally we have to show universality of $\tau$ in $\Ab$. Let $\mu_O:X(O)\to B$ be any other cone, i.e., each $\mu_O\in\Ab$ and $\mu_P\ci X(t)=\mu_O$ for all $t:O\to P$. Then the map $f:A\to B$, $f([O,x])=\mu_O(x)$ is
well-defined, additive and satisfies $f\ci\tau_O=\mu_O$ for all object $O$ by construction. If $f':A\to B$ also
satisfies $f'\ci\tau_O=\mu_O$ for all $O$ then $f'([O,x])=f'\ci\tau_O(x)=\mu_O(x)=f\ci\tau_O(x)=f([O,x])$, hence $f'=f$.
\end{proof}

Applying the Lemma to the case $\I=(\El \Upsilon)^\op$ and $X= F\phi$ we obtain a description of the Abelian group
$\colim F\phi$ as the set $\U(\colim F\phi)=\colim\U F\phi$ endowed with a natural addition rule.
But how is $\colim F\phi$ related to the tensor product?

Every $\Ab$-valued additive functor on the small category $\C^\op$ is the colimit of representables and
in a canonical way. This means that $\Upsilon$ is the colimit
\[
\Upsilon=\colim\left((\El\Upsilon)^\op\rarr{\phi}\C\rarr{Y}\Add(\C^\op,\Ab)\right)
\]
with universal cone
\begin{align*}
\Omega_{\bra y,\alpha\ket}&:\C(\under,\alpha)\to\Upsilon\\
\Omega_{\bra y,\alpha\ket,\beta}&:\C(\beta,\alpha)\to\Upsilon(\beta),\quad t\mapsto\Upsilon(t)y\,.
\end{align*}
Since $\under\amo{\C}F$ preserves colimits, the (\ref{rep tensor F}) implies
\[
\Upsilon\amo{\C}F=\colim\left((\El\Upsilon)^\op\rarr{\phi}\C\rarr{F}\Ab\right)\,.
\]
This proves that the tensor product can be presented as a set of elements $y\amo{\alpha}x$
redundantly labelled by $y\in\Upsilon(\alpha)$, $x\in F(\alpha)$ and $\alpha\in\ob\C$ with identifications
(\ref{tensor rels}) and with addition rule
\begin{equation} \label{tensor addition general}
y_1\amo{\alpha}x_1 + y_2\amo{\beta}x_2 
= (\Upsilon(p_1)y_1+\Upsilon(p_2)y_2)\amo{\sigma}(F(i_1)x_1 + F(i_2)x_2)
\end{equation}
where $\sigma$ and $p_1,p_2,i_1,i_2$ are as in (\ref{tensor addition}).

There are other situations where the tensor product can be computed set theoretically.
Let us abandon the assumption that direct sums exist. In its stead assume that $\Upsilon$ is flat, i.e.,
that the category of elements $\El\Upsilon$ is cofiltered. (For notions like filtered, cofiltered, flat we use
the terminology of \cite{Borceux}). Then there is a well known result \cite[1.10]{TTT}  which replaces Lemma
\ref{lem: U}. It states that colimits of functors $\I\rarr{X}\Ab\rarr{\U}\Set$, with $\I$ small, filtered and $X$ additive,
are underlying a colimit of $X$. Since then $\Upsilon\amo{\C}F$ is a filtered colimit, it can be identified with the
set of elements $y\amo{\alpha}x$ up to the above identifications and with appropriate addition rule.

Alternatively, we may assume that $F$ is flat. Since the tensor product is also a colimit
\[
\Upsilon\amo{\C}F=\colim\left((\El F)^\op\rarr{}\C^\op\rarr{\Upsilon}\Ab\right)\,,
\]
we obtain the same conclusion.
\begin{cor}
Let $\C$ be a small $\Ab$-category and $\Upsilon:\C^\op\to\Ab$, $F:\C\to\Ab$ be additive functors.
In either one of the cases
\begin{enumerate}
\item $\C$ has binary direct sums
\item $\Upsilon$ is flat
\item $F$ is flat
\end{enumerate}
the tensor product $\Upsilon\amo{\C}F$ defined by (\ref{tensor}) can be presented as the set of connected
components of the graph
\begin{align*}
\ob \Gamma&=\{\bra y,\alpha,x\ket\,|\,y\in\Upsilon(\alpha), x\in F(\alpha),\alpha\in\ob\C\,\}\\
\Gamma(\bra y,\alpha,x\ket,\bra y',\beta,x'\ket)&=\{t\in\C(\alpha,\beta)\,|\,\Upsilon(t)y'=y,F(t)x=x'\,\}
\end{align*}
equipped with addition rule
\begin{enumerate}
\item 
$[y_1,\alpha,x_1]+[y_2,\beta,x_2]=[\Upsilon(p_1)y_1+\Upsilon(p_2)y_2,\sigma,F(i_1)x_1 + F(i_2)x_2]$
where $\alpha \overset{p_1}{\underset{i_1}{\leftrightarrows}} \sigma
\overset{p_2}{\underset{i_2}{\rightleftarrows}}\beta$ is a direct sum diagram in $\C$
\item
$[y_1,\alpha,x_1]+[y_2,\beta,x_2]=[u,\sigma,F(i_1)x_1 + F(i_2)x_2]$
where $\alpha\rarr{i_1}\sigma\larr{i_2}\beta$ and $u\in\Upsilon(\sigma)$ are such that $\Upsilon(i_1)u=y_1$
and $\Upsilon(i_2)u=y_2$
\item
$[y_1,\alpha,x_1]+[y_2,\beta,x_2]=[\Upsilon(p_1)y_1+\Upsilon(p_2)y_2,\sigma,z]$
where $\alpha\larr{p_1}\sigma\rarr{p_2}\beta$ and $z\in F(\sigma)$ are such that $F(p_1)z=x_1$
and $F(p_2)z=y_2$
\end{enumerate}
in the respective cases.
\end{cor}

Returning to our original problem, flatness of the fiber functors $F$ would be a natural assumption if we were
doing Tannaka reconstruction since flatness is the natural replacement for left exactness when the category in question, $\C$, is lacking (all) finite limits.
Flatness of $\Upsilon$, on the other hand, means the following 2 conditions on the ring $B$ and on the subcategory
$\C\subset\Ring(B,B)$ the first of which being automatically satisfied if $\C$ has direct sums:
\begin{itemize}
\item For all $a,b\in B$, $\alpha,\beta\in\ob\C$  there exist $c\in B$ and $\alpha\rarr{t}\gamma\larr{s}\beta$ in $\C$ such that $ct=a$ and $cs=b$.
\item For all $b\in B$ and $t:\gamma\to\beta$ in $\C$ such that $bt=0$ there exist $a\in B$ and $e:\beta\to\alpha$
in $\C$ such that $et=0$ and $ae=b$.
\end{itemize}

Finishing the Subsection we return to the general case where neither the existence of direct sums nor flatness of $\Upsilon$ or $F$ is assumed.
From now on $\Upsilon$ always denotes the basic presheaf of Proposition \ref{pro: Ups}.

\subsection{The forgetful functor of an extension} \label{ss: rho>ff}

For a ring morphism $\rho:B\to A$ and for each endomorphism $\alpha:B\to B$ we define two additive subgroups of $A$
\begin{align}
F(\alpha)&:=\Hom(\rho,\rho\alpha)\ \equiv\ \{a\in A\,|\,a\rho(b)=\rho\alpha(b)a,\ b\in B\}\\
\bar F(\alpha)&:=\Hom(\rho\alpha,\rho)\ \equiv\ \{a\in A\,|\,a\rho\alpha(b)=\rho(b)a,\ b\in B\}\,.
\end{align}
One can easily recognize the $F$-s as the analogues of the `Hilbert spaces in the field algebra' of the
Doplicher-Roberts theory: The elements of $F(\alpha)$ are said to `create charge $\alpha$' and those of
$\bar F(\alpha)$ to `annihilate' it.

\begin{lem} \label{lem: F}
For a morphism $\rho:B\to A$ in $\Ring$ the correspondence $\alpha\mapsto F(\alpha)$ is the object map of a
functor $F:\Ring(B,B)\to\Ab$ sending the arrow $t:\alpha\to\beta$ to the map $F(t):F(\alpha)\to F(\beta)$,
$a\mapsto \rho(t)a$.
This functor has a monoidal structure given by the natural transformation $F_{\alpha,\beta}:F(\alpha)\o F(\beta)\to F(\alpha\beta)$, $a\o a'\mapsto aa'$ and by the map $F_0:\ZZ\to F(\iota)\equiv R$, $1\mapsto 1_R$.
\end{lem}
\begin{proof}
This is obvious.
\end{proof}
Notice that multiplication in $A$ either from the left or right by elements of $R=\End\rho$ leaves $F(\alpha)$ (and
$\bar F(\alpha)$) invariant. What is more, the functor $F$ factorizes through $_R\M_R$ as a monoidal functor.
In fact this is not an accident. This is just the canonical factorization of a monoidal functor \cite{Sz: Brussels} since
$R$, as a monoid in $\Ab$, is the image under $F$ of the trivial monoid $\iota$ of $\C$ and the left and right multiplications by elements of $R$ coincide with the left and right actions $F_{\alpha,\iota}$, $F_{\iota,\alpha}$.

The restrictions of $F$ to certain subcategories $\C$ are particularly nice. 
\begin{defi}
An object $\sigma\in\Ring(B,B)$ is called implementable in the extension $\rho:B\to A$ if there exist $f^1,\dots,f^n\in \Hom(\rho,\rho\sigma)$ and $\bar f^1,\dots,\bar f^n\in\Hom(\rho\sigma,\rho)$ such that
$$
\rho(\sigma(b))=\sum_i\ f^i\,\rho(b)\,\bar f^i\,,\qquad b\in B\,.
$$
\end{defi}

\begin{lem}
$\sigma\in\ob\Ring(B,B)$ is implementable in the extension $\rho:B\to A$ if and only if $\rho\sigma\leq\rho$ where $\leq$ is the preorder introduced in Subsection \ref{ss: preorder}.
\end{lem}
\begin{proof}
The implementability relation is equivalent to $\sum_i f^i\bar f^i=1_A$ which in fact means that $\rho\sigma\rarr{\bar f^i}\rho\rarr{f^i}\rho\sigma$ is a direct summand diagram.
\end{proof}
\begin{lem&def}
Let $\Imp(\rho)$ denote the full subcategory of $\Ring(B,B)$ the objects of which are the endomorphisms that
are implementable in $\rho$. Then $\Imp(\rho)$ is a monoidal subcategory of $\Ring(B,B)$.
\end{lem&def}
\begin{proof}
Clearly, $\iota$ is implementable. If $\rho\sigma\leq\rho$ and $\rho\tau\leq\rho$ then $\rho\sigma\tau\leq\rho\tau\leq\rho$.
\end{proof}

\begin{pro} \label{the good F}
For a morphism $\rho:B\to A$ let $\C$ be any full monoidal subcategory of $\Imp(\rho)$. Then the restriction of the
functor $F=\Hom(\rho,\rho\under)$ to $\C$ is an essentially
strong monoidal functor $\C\to\Ab$ such that $F(\sigma)$ is finitely generated projective as right module over $R=F(\iota)$ for each $\sigma\in\ob\C$.
\end{pro}
\begin{proof}
$F$ is the restriction to $\C$ of the monoidal functor defined in Lemma \ref{lem: F}.
Therefore $F(\alpha)=\Hom(\rho,\rho\alpha)$ has bimodule structure
$r\cdot x\cdot r'=rxr'$ which is finite projective from the right due to $\rho\sigma\leq\rho$ and Lemma \ref{Morita} (1).
Essential strongness of $F$ means that its canonical factorization through $_R\M_R$ is a strong monoidal functor
$\C\to\,_R\M_R$. The unit of this functor is the identity map $R\to F(\iota)$ by construction. Invertibility of $F(\alpha)\oR F(\beta)\to F(\alpha\beta)$ can be best seen by explicitly constructing its inverse:
\[
\Hom(\rho,\rho\alpha\beta)\ni z\quad\mapsto\quad \sum_i q_i\oR p_iz
\]
where $\rho\alpha\rarr{p_i}\rho\rarr{q_i}\rho\alpha$ is any direct summand diagram corresponding to the
assumption $\rho\alpha\leq\rho$.
\end{proof}

\subsection{The adjunction $\mathcal{E}\dashv\mathcal{F}$}

In this subsection we fix a ring $B$ and a full monoidal subcategory $\C\subset\Ring(B,B)$ and will show that
the construction of extensions of $B$ via fiber functors described in Subsection \ref{ss: ff>rho} is left adjoint to the construction of fiber functors from extensions as we described in Subsection \ref{ss: rho>ff}.

Let $\Bslice$ be the category  the objects of which are the morphisms $\rho:B\to A$ for some ring $A$ and the arrows
$(A,\rho)\to(A',\rho')$ are the morphisms $\kappa:A\to A'$ such that $\kappa\rho=\rho'$.
$\MonFun(\C,\Ab)$ denotes the category consisting of additive monoidal functors from $\C$ to $\Ab$ as objects and monoidal natural transformations between them as arrows.

The construction of extensions from functors can be described as the functor
\begin{align}
\mathcal{E}:\MonFun(\C,\Ab)&\to\Bslice\\
F&\mapsto(\Upsilon\amalgo{\C}F,\rho_F)\notag\\
(F\rarr{\nu}F')&\mapsto\{b\amo{\alpha}x\mapsto b\amo{\alpha}\nu_\alpha x\}\notag
\end{align}
where $\rho_F:B\to\Upsilon\amalgo{\C}U$ is the morphism defined in (\ref{eq: B->F}).
The construction of functors from extensions in turn is the functor
\begin{align}
\mathcal{F}:\Bslice&\to\MonFun(\C,\Ab)\\
(A,\rho)&\mapsto F=\{\alpha\mapsto\Hom(\rho,\rho\alpha)\}\notag\\
((A,\rho)\rarr{\kappa}(A',\rho'))&\mapsto\{\nu_\alpha:a\mapsto\kappa(a)\}\,.\notag
\end{align}
\begin{thm}
Let $B$ be a ring and let $\C$ be a full monoidal subcategory of $\Ring(B,B)$. Then the functor $\mathcal{E}$ is left
adjoint to the functor $\mathcal{F}$. The adjunction is given by the isomorphism
\begin{align*}
\Bslice(\mathcal{E}F,\rho)&\iso\MonFun(\C,\Ab)(F,\mathcal{F}\rho)\\
\parbox{60pt}{
\begin{picture}(60,50)
\put(26,38){$B$}
\put(25,35){\vector(-2,-3){15}} \put(5,25){$\rho_F$}
\put(33,35){\vector(2,-3){15}} \put(44,25){$\rho$}
\put(-10,2){$\Upsilon\amalgo{\C}F$}
\put(20,4){\vector(1,0){20}} \put(27,8){$\kappa$}
\put(45,2){$A$}
\end{picture}
}
&\mapsto \left\{F(\alpha)\ni x\mapsto
\kappa(1_B\amalgo{\alpha}x)\in\Hom(\rho,\rho\alpha)\right\}_{\alpha\in\ob\C}
\end{align*}
of Abelian groups natural in $F\in\MonFun(\C,\Ab)$ and in $\rho\in\Bslice$.
\end{thm}
\begin{proof}
The inverse associates to a monoidal natural transformation $\nu_\alpha:F(\alpha)\to\Hom(\rho,\rho\alpha)$ the
morphism of extensions $\kappa:\Upsilon\amalgo{\C}F\to A$, $b\amo{\alpha}x\mapsto \rho(b)\nu_\alpha x$.
\end{proof}

The components of the unit of the adjunction is the natural transformation $\eta_F:F\to\mathcal{F}\mathcal{E}F$ with components
$$
(\eta_F)_\alpha:F(\alpha)\to\Hom(\rho_F,\rho_F\alpha)\,,\quad x\mapsto 1_B\amo{\alpha}x\,.
$$

The counit of the adjunction  $\epsilon_{(A,\rho)}:\mathcal{E}\mathcal{F}(A,\rho)\to(A,\rho)$ is the morphism
$$
\Upsilon\amalgo{\C}\mathcal{F}(A,\rho)\to A\,,\quad b\amalgo{\C}t\mapsto \rho(b)t\,.
$$

Now we can formulate a universal property of the field algebra extension $\E(F)$ of $F:\C\to\Ab$
but only together with the natural map $\eta_F$ that relates it to the functor $\F$. 
Given the construction $\F$ of fiber functors from extensions of $B$ and a functor $F$ a universal arrow \cite{CWM}
from $F$ to $\F$ consists of an extension $\rho:B\to A$ and a natural transformation $\eta:F\to \F\rho$ (mapping the elements of $F$ to elements of $A$ creating the appropriate charge) such that if
$\rho':B\to A'$ is any other extension equipped with a natural transformation $\nu:F\to \F\rho'$ then there is a unique morphism $\kappa:\rho\to \rho'$ of $B$-rings such that $\F\kappa\ci\eta=\nu$. Of course, such a universal
arrow $\bra \rho,\eta\ket$ is necessarily isomorphic to $\bra \rho_F,\eta_F\ket$ defined before.

From the point of view of the adjunction $\E\dashv\F$ the `good' functors $F:\C\to\Ab$ are the ones for which $\eta_F$ is an isomorphism and the `good' extensions $\rho:B\to A$ are the ones for which $\eps_\rho$ is an isomorphism.
The full subcategories of $\MonFun(\C,\Ab)$ and $\Bslice$ of `good' functors and extensions, respectively, are
equivalent categories by restricting/corestricting $\E$ and $\F$. Unfortunately, the `good' functors and extensions
seem to be too difficult to analyze in this generality and we have to select a more tractable case of fiber functors.

\subsection{Field algebras with bialgebroid cosymmetry} \label{ss: fa with bgd}

Let a morphism $\rho:B\to A$ be fixed and consider the functor $F=\Hom(\rho,\rho\under):\C\to \Ab$ studied
in Subsection \ref{ss: rho>ff} where $\C$ is a full monoidal subcategory of $\Imp(\rho)$.
We have seen that $F$ is essentially strong with strong monoidal part $\hat F:\C\to\,_R\M_R$ mapping each object $\alpha$ into a bimodule $\hat F(\alpha)=\bra F(\alpha),F_{\iota,\alpha},F_{\alpha,\iota}\ket$ which is finitely generated
projective as a right $R$-module.
Let $F^*:\C^\op\to\,_R\M_R$ be defined by $F^*(\alpha):=\Hom(\rho\alpha,\rho)\cong\Hom_{-R}(\hat F(\alpha),R)$
and let it be given the strong monoidal structure
\begin{align*}
F^*_{\alpha,\beta}&:F^*(\alpha)\amo{R}F^*(\beta)\to F^*(\beta\alpha),\quad f\amo{R}g\mapsto fg\\
F^*_0&:R\to F^*(\iota),\qquad r\mapsto r\,.
\end{align*}
that makes $F^*$ a strong monoidal functor from $\C^\op$ to $_R\M_R^\rev$, the latter denoting $_R\M_R$ with reversed (i.e., opposite) monoidal structure. Now we can define the `long dual' $F^\lef:\C^\op\to \Ab$ as the
composite of monoidal functors
\[
\C^\op\rarr{F^*}\,_R\M_R^\rev\rarr{}\Ab^\rev\rarr{\Sigma}\Ab
\]
where $\Sigma$ is the identity functor with monoidal structure consisting of the symmetry $X\o Y\to Y\o X$ and the  identity arrow $\ZZ\to \ZZ$. This implies that the functor $F^\lef$
has monoidal structure
\begin{align*}
F^\lef_{\alpha,\beta}&:F^\lef(\alpha)\amo{R}F^\lef(\beta)\to F^\lef(\alpha\beta),\quad f\o g\mapsto gf\\
F^\lef_0&:\ZZ\to R,\qquad 1\mapsto 1_R\,.
\end{align*}
\begin{pro} \label{pro: bgd cosym}
With the above notations the tensor product $H:=F^\lef\amo{\C}F$ is a right bialgebroid over $R$
with multiplication rule
\[
(f\amo{\alpha}x)(g\amo{\beta}y)=gf\amo{\alpha\beta}xy
\]
with source and target maps
\begin{align*}
s_H&:R\to H,\quad r\mapsto 1_{R^\op}\amo{\iota}r\\
t_H&:R^\op\to H,\quad r\mapsto r\amo{\iota}1_R
\end{align*}
and coring structure
\begin{align*}
\cop_H&:H\to H\amo{R}H,\quad (f\amo{\alpha}x)\mapsto\sum_i(f\amo{\alpha}x_\alpha^i)\amo{R}
(f_\alpha^i\amo{\alpha} x),\\
\eps_H:&H\to R,\qquad (f\amo{\alpha}x)\mapsto fx\,,
\end{align*}
where $\{x_\alpha^i,f_\alpha^i\}$ for $\alpha\in\ob\C$ denote (arbitrary) dual bases for the dual modules
$F(\alpha)_R$ and $_RF^\lef(\alpha)$, or, equivalently, direct summand diagrams $\rho\alpha\rarr{f^i_\alpha}\rho\rarr{x^i_\alpha}\rho\alpha$.
Furthermore, the field algebra $\Upsilon\amo{\C}F$ associated to $F$ is a right $H$-comodule
algebra with underlying $R$-ring structure given by $\pi_F$ and with $\rho_F(B)$ contained in the coinvariant subalgebra.
\end{pro}
\begin{proof}
Checking of the right bialgebroid axioms as given, e.g., in \cite{Ka-Sz} is a routine calculation. The $H$-coaction
on the field algebra $\A=\Upsilon\amo{\C}F$ 
\begin{equation} \label{delta}
\delta_\A:\A\to\A\amo{R}H,\quad (b\amo{\alpha}x)\mapsto\sum_i(b\amo{\alpha}x_\alpha^i)\amo{R}(f_\alpha^i\amo{\alpha}x)
\end{equation}
is such that its image factors through the Takeuchi product $\A\ex{R}H$ which is the subbimodule
\[
\A\ex{R}H=\{\sum_ia_i\oR h_i\in\A\oR H\,|\,\sum_ira_i\oR h_i=\sum_i a_i\oR t_H(r)h_i,\forall r\in R\,\}\,.
\]
It is now meaningful to ask multiplicativity of the map $\delta_\A$ and the answer is affirmative by the calculation
\begin{align*}
\delta_\A(b\amo{\alpha}x)\delta_\A(b'\amo{\beta}x')&=
\sum_i\sum_j(b\amo{\alpha}x_\alpha^i)(b'\amo{\beta}x_\beta^j)\oR
(f_\alpha^i\amo{\alpha}x)(f_\beta^j\amo{\beta}x')\\
&=\sum_i\sum_j\left(b\alpha(b')\amo{\alpha\beta}x_\alpha^ix_\beta^j\right)\oR
\left(f_\beta^jf_\alpha^i\amo{\alpha\beta}xx'\right)\\
&=\sum_k(b\alpha(b')\amo{\alpha\beta}x_{\alpha\beta}^k)\oR(f_{\alpha\beta}^k\amo{\alpha\beta}xx')\\
&=\delta_\A((b\amo{\alpha}x)(b'\amo{\beta}x'))\,.
\end{align*}
Since $\delta_\A$ preserves the unit obviously, we have shown that $\bra\A,\delta_\A\ket$ is an $H$-comodule algebra. Also, one has $\delta_\A(\rho_F(b))=(b\amo{\iota}1_R)\oR(1_{R^\op}\amo{\iota}1_R)=\rho_F(b)\oR 1_H$
which finishes the proof.
\end{proof}

We note that since the tensor product $F^\lef\amo{\C}F$ is also the coend of $F^\lef\o F$,
the bialgebroid $H$ of Proposition \ref{pro: bgd cosym} is nothing but an example of the
bialgebroid constructed in \cite{Phung} for general fiber functors.

\section{Adjoint morphisms}

The content of this section arises from applying basic categorical constructions to $\Ring$ producing
a very elementary but still interesting class of ring extensions. 

\subsection{Generators for a comonad in $\Ring$} \label{ss: generators of gamma}
A comonad in $\Ring$ consists of a ring $B$ and a comonoid $\bra\gamma,d,e\ket$ in $\Ring(B,B)$.
The latter means that
$\gamma:B\to B$ is a ring endomorphism and $d:\gamma\to\gamma^2$ and $e:\gamma\to B$ are intertwiners satisfying (as elements of $B$) the relations: $\gamma(d)d=d^2$, $ed=1_B=\gamma(e)d$.

The Kleisli category for the comonad $\bra B,\gamma,d,e\ket$ has a single object since $B$ has only one.
Its arrows are all the arrows (i.e., elements) of $B$ but the composition is different. It is the associative operation
$$
b_1\star b_2:=b_1\gamma(b_2)d
$$
for which $e$ serves as the unit. Then $B_\gamma:=\bra B,\star,e\ket$ is a ring and $b\mapsto be$ is a ring homomorphism $\rho_\gamma:B\to B_\gamma$. It is right adjoint to the Kleisli functor $\lambda_\gamma: b\mapsto \gamma(b)d$ with
$m:=1_B$, as an element of $B_\gamma$, as the unit and $e\in B$ as the counit.
As a matter of fact,
\begin{align*}
\rho_\gamma\lambda_\gamma(b')\star m&=\gamma(b')de\gamma(1_B)d=
\gamma(b')d=m\star b',\quad b'\in B_\gamma\\
e\lambda_\gamma\rho_\gamma(b)&=e\gamma(be)d=e\gamma(b)=be,\qquad b\in B\\
\rho_\gamma(e)\star m&=ee\gamma(1_B)d=e=1_{B_\gamma}\\
e\lambda_\gamma(m)&=ed=1_B
\end{align*}
In this way we have shown that every comonad on $B$ arises from an adjunction
$m,e:\lambda_\gamma\dashv\rho_\gamma$ as 
\begin{equation}
\bra\gamma,d,e\ket=\bra \lambda_\gamma\rho_\gamma,\lambda_\gamma(m),e\ket\,.
\end{equation}

\subsection{The Kleisli comparison morphism in $\Ring$}
Let an adjunction $\lambda\dashv \rho:B\to A$ be given in $\Ring$ with counit
$e:\lambda\rho\to B$ and unit $m:A\to \rho\lambda$. Then we can construct the comonad $\gamma:=\lambda\rho$ with comultiplication $d=\lambda(m)$ and counit $e$. The associated Kleisli extension $\rho_\gamma:B\to B_\gamma$, $b\mapsto be$,
is compared to the original extension by the morphism $\kappa:B_\gamma\to A$,
$\kappa(b):=\rho(b)m$.
$$
\parbox{100pt}{
\begin{picture}(100,80)
\put(10,10){$B$}
\put(20,10){\vector(1,0){40}} \put(38,4){$\sst\rho_\gamma$}
\put(60,14){\vector(-1,0){40}} \put(38,17){$\sst\lambda_\gamma$}
\put(65,10){$B_\gamma=\bra B,\star,e\ket$}
\put(10,60){$A$}
\put(15,20){\vector(0,1){35}} \put(18,35){$\sst\rho$}
\put(12,55){\vector(0,-1){35}} \put(5,35){$\sst\lambda$}
\put(60,20){\vector(-1,1){38}} \put(44,40){$\sst\kappa$}
\end{picture}
}
$$
This morphism satisfies $\lambda\kappa(b)=\gamma(b)d=\lambda_\gamma(b)$ and $\kappa\rho_\gamma(b)=\rho(b)$ which are the defining
properties of the comparison functor. Since $\kappa$ is always invertible by $a\mapsto e\lambda(a)$, any right adjoint morphism $\rho:B\to A$ is isomorphic to a Kleisli extension of $B$.

\subsection{Cross product with a comonad}

For a ring $B$ and a comonad $\bra\gamma,d,e\ket$ on $B$ we can construct a morphism 
$\rho:B\to B\cross{}\gamma$ as follows. Let $B\cross{}\gamma$ be the rng generated by
the ring $B$ and by a single additional generator $m$ subject to the relations
\begin{align}
\label{rel1}
mb&=\gamma(b)m,\qquad b\in B\\
\label{rel2}
m\,1_B&=m\\
\label{rel3}
m^2&=dm\\
\label{rel4}
em&=1_B\,.
\end{align}
Let $\rho:B\to B\cross{}\gamma$ be the map sending $b\in B$ to the generator called $b$.
We claim that $B\cross{}\gamma$ is unital and $\rho$ is a ring homomorphism.
As a matter of fact, $1:=1_B$ is central since $\gamma(1_B)=1_B$ and satisfies $b1=b$,  $b\in B$ and $m1=m1_B=m$. Therefore $\rho$ is a morphism, indeed. 

It follows from (\ref{rel1}) and (\ref{rel3}) that every element of $B\cross{}\gamma$ can be written as $bm$ for some $b\in B$. For example, $1=em$. In other words, the map
$b\mapsto\rho(b)m$ is an epimorphism $f:B\to B\cross{}\gamma$ as left $B$-modules. In fact it is an isomorphism since $a\to ad$ is a left inverse of $f$. It is not difficult to see now that
under this isomorphism the ring structure of $B\cross{}\gamma$ is mapped to that of $B_\gamma$. Therefore the crossed product $B\cross{}\gamma$ is nothing but another
presentation of the Kleisli extension.

\subsection{Left regular extensions}
Let us call a morphism $\rho:B\to A$ left regular if the induced module $_BA$ is isomorphic to the left regular module $_BB$. For a left regular $\rho$ choose an isomorphism $f:\,_BB\to\,_BA$ and define
$m:=f(1_B)$ and $e:=f^{-1}(1_A)$. Then $f(b)=\rho(b)m$ and therefore $\rho(e)m=1_A$. 
It follows that $\lambda:A\to B$, $\lambda(a):=f^{-1}(ma)$ is an algebra map,
\[
\lambda(a)\lambda(a')=f^{-1}(\rho(f^{-1}(ma))ma')=f^{-1}(f(f^{-1}(ma))a')
=f^{-1}(maa')=\lambda(aa')\,,
\]
and satisfies $ma=\rho\lambda(a)m$ for $a\in A$. We obtain that $f^{-1}(a)=f^{-1}(\rho(e)ma)=e\lambda(a)$ and in particular $e\lambda\rho(b)=f^{-1}(\rho(b))=be$. Also, since $f^{-1}(m)=1_B$, $e\lambda(m)=1_B$.
This proves that $\rho$ is right adjoint to $\lambda$ with unit $m$ and counit $e$.

As a byproduct $f$ becomes also a $B$-$A$-bimodule isomorphism
\begin{equation}
f:\,_BB_A\iso\,_BA_A,\qquad b\mapsto\rho(b)m
\end{equation}
in the sense of $f(b_1b\lambda(a))=\rho(b_1)f(b)a$, $b_1,b\in B,\ a\in A$.

In the above Subsections we have characterized a special class of ring extensions in several ways.
\begin{thm} \label{thm: adjoint ext}
For a ring homomorphism $\rho:B\to A$ the following conditions are equivalent:
\begin{enumerate}
\item $\rho:B\to A$ is the Kleisli construction for some comonad on $B$.
\item $\rho:B\to A$ is the crossed product $B\cross{}\gamma$ for some comonad on $B$.
\item $_BA\cong\, _BB$.
\item There is a morphism $\lambda:A\to B$ w.r.t. which $_BA_A\cong\,_BB_A$.
\item $\rho$ is right adjoint in $\Ring$.
\end{enumerate}
\end{thm}

\begin{cor}
If $\lambda\dashv\rho$ is an adjunction in $\Ring$ then both $\lambda$ and $\rho$ are monomorphisms.
If $\bra \gamma,d,e\ket$ is a comonoid in $\Ring(B,B)$ then $\gamma$ is a monomorphism.
\end{cor}

\subsection{Comonadic ring extensions}
Viewing $\Ring$ as a full sub-2-category in $\Ab$-$\Cat$ there is a construction of the Eilenberg-Moore category
$B^\gamma$ for each comonad $\gamma$ on $B$, although it may lie outside $\Ring$.

For a comonad $\bra B,\gamma,d,e\ket$ the objects of $B^\gamma$ are the elements $x\in B$ satisfying
$\gamma(x)x=dx$ and $ex=1_B$. The arrows $r:x\to y$ are the $r\in B$ for which $\gamma(r)x=yr$. Composition is multiplication in $B$. 
The endomorphism ring of the object $d\in B^\gamma$, as a 1-object category, is isomorpic to the Kleisli category  $B_\gamma$. Indeed,
\[
\varphi:B_\gamma\to B^\gamma(d,d),\qquad b\mapsto \gamma(b)d
\]
is well-defined, $\gamma(\gamma(b)d)d=\gamma^2(b)d^2=d\gamma(b)d$, it is a functor since $\gamma(b_1)d\gamma(b_2)d=$
$\gamma(b_1\gamma(b_2)d)d$ and $\gamma(e)d=1_B$, it is injective, $r=\gamma(b)d\ \Rightarrow\ er=b$, and surjective since
$\gamma(r)d=dr$ multiplied with $\gamma(e)$ from the left yields $\gamma(er)d=r$.

If $\lambda\dashv \rho:B\to A$ is an arbitrary adjunction then the Eilenberg-Moore comparison functor maps $A$ to
the category $B^\gamma$ associated to the comonad $\gamma=\lambda\rho$. It maps the single object of $A$ to
$\lambda(m)=d$ and the arrows $a\in A$ to the arrows $r=\lambda(a)$. This functor $\kappa^\gamma:A\to B^\gamma$ is the composite of $\kappa_\gamma^{-1}:A\to B_\gamma$ and $\varphi:B_\gamma\to B^\gamma(d,d)$. As a matter of fact, $\gamma(e\lambda(a))d=\lambda(\rho(e\lambda(a))m)=\lambda(a)$. Since $\varphi\kappa_\gamma^{-1}$ is always an isomorphism, the comparison functor is an equivalence of categories, in which case one says $\rho$ (or rather $\lambda$) is comonadic, iff all objects of $B^\gamma$ are isomorphic. That is to say, $\rho$ is comonadic iff the equations
\begin{align*}
\gamma(x)x&=dx\\
ex&=1_B
\end{align*}
for $x\in B$ have the general solution $x=\gamma(u^{-1})u$ with $u\in B$ invertible.

\section{Field algebras that are Galois extensions}

Adjoint extensions with the extra condition of depth 2 lead to Galois extensions over bialgebroids.
This fact has already been recognized in an increasing level of depth in \cite{Sz: Siena,Ka-Sz,Ka: Gal bgd}.
Here we recapitulate, extend or specialize some of those arguments with the focus of attention being the structure of the field algebra.

\subsection{Representable fiber functors} \label{ss: rep f}

Let $\bra \gamma,d,e\ket$ be a comonad on $B$. Then $\Hom(\gamma,\under)$ is a representable monoidal functor $\Ring(B,B)\to\Ab$. Let $\C$ be any full monoidal subcategory
of $\Ring(B,B)$ which contains the object  $\gamma$. Then, as a counterpart of formula
(\ref{rep tensor F}), we obtain that
 the field algebra $A=\Upsilon\amo{\C}F$ associated to the restriction $F:\C\to\Ab$ of the above functor is isomorphic to $\Upsilon(\gamma)=B$ by
\[
b\amo{\alpha}x=b\amo{\alpha}F(x)1_\gamma=bx\amo{\gamma}1_\gamma\quad
\mapsto \quad bx
\]
which is clearly a left $B$-module map. Therefore the field algebra extension $\rho_F:B\to A$ is isomorphic to the Kleisli construction $\rho_\gamma:B\to B_\gamma$ via the map
$\kappa^{-1}:A\to B_\gamma$, $b\amo{\alpha}x\mapsto bx$. In order to check
multiplicativity of this map use (\ref{multip}) and the the formula
$F_{\alpha,\beta}(x\o y)=x\gamma(y)d$ for the monoidal structure of $\Hom(\gamma,\under)$
to obtain
\begin{align*}
\kappa^{-1}((b\amo{\alpha}x)(b'\amo{\beta}x'))&=\kappa^{-1}(b\alpha(b')\amo{\alpha\beta}x\gamma(x')d)=\\
&=b\alpha(b')x\gamma(x')d=bx\gamma(b'x')d=(bx)\star(b'x')=\\
&=\kappa^{-1}(b\amo{\alpha}x)\star\kappa^{-1}(b'\amo{\beta}x')\,.
\end{align*}
We note that the inverse, which is the Kleisli comparison functor, is given by the formula $\kappa:B_\gamma\to A$, $b\mapsto b\amo{\gamma}1_\gamma$.
\begin{lem} \label{lem: rep}
For $B$ a ring, $\C$ a full monoidal subcategory of $\Ring(B,B)$ and $\bra\gamma,d,e\ket$ a comonoid in $\C$
the representable monoidal functor $F:=\Hom(\gamma,\under):\C\to\Ab$ gives an example of an $F\in\Mon\Fun(\C,\Ab)$ such that
\begin{enumerate}
\item the extension $\mathcal{E}(F)=B\rarr{\rho_F}A$ is right adjoint and
\item the unit of $\mathcal{E}\dashv\mathcal{F}$ at $F$ is an isomorphism $\eta_F:F\iso\mathcal{F}\mathcal{E}(F)$.
\end{enumerate}
\end{lem}
\begin{proof}
The first part has been shown above.
Proving the second part notice that $F$ is isomorphic to $\mathcal{F}(\rho_\gamma)$ since comonads have generators
by Subsection \ref{ss: generators of gamma}. The natural isomorphism $\Hom(\gamma,\under)=\Hom(\lambda_\gamma\rho_\gamma,\under)
\iso\Hom(\rho_\gamma,\rho_\gamma\under)$ sending $x$ to $\rho_\gamma(x)\star 1_B=x$ is the restriction
to the hom-groups $\Hom(\gamma,\alpha)\subset B$ of $\id_B$ considered as a map $B\to B_\gamma$.
Composing with the Kleisli comparison functor $B_\gamma\to A$ we have the sequence of isomorphisms
\[
F(\alpha)=\Hom(\gamma,\alpha)=\Hom(\lambda_\gamma\rho_\gamma,\alpha)\iso
\Hom(\rho_\gamma,\rho_\gamma\alpha)\rarr{\kappa}\Hom(\rho_F,\rho_F\alpha)
\]
sending $x$ to $\kappa(x)$. Since for $x\in\Hom(\gamma,\alpha)$ we can write $\kappa(x)=x\amo{\gamma}1_\gamma=1_B\amo{\alpha}x$, this isomorphism is precisely the $\alpha$-component
of the unit $\eta_F$.
\end{proof}

\subsection{Regular Galois extensions}

The representable monoidal functor of the previous subsection is not essentially strong in general, neither has it any quantum groupoid symmetry. Here we would like to unite the benefits of the functors studied in Subsections
\ref{ss: fa with bgd} and \ref{ss: rep f}.
Since every representable monoidal functor is of the canonical form $F=\mathcal{F}(\rho)=\Hom(\rho,\rho\under)$
for some morphism $\rho:B\to A$ by Lemma \ref{lem: rep}, we are left with finding appropriate properties on the morphism $\rho$.

In Subsection \ref{ss: fa with bgd} we assumed $\C$ to be a full monoidal subcategory of $\Imp(\rho)$, in Subsection \ref{ss: rep f} $\rho$ had a left adjoint $\lambda$
satisfying $\lambda\rho\in\C$. Therefore we shall consider here the situation of an adjunction $\lambda\dashv\rho$ such that $\rho\lambda\rho\leq\rho$.

\begin{lem} \label{lem: D2}
For a morphism $\rho:B\to A$ with a left adjoint $\lambda$ in $\Ring$ the
following conditions are equivalent.
\begin{enumerate}
\item $\rho$ is a right depth 2 extension \cite{Ka-Sz}.
\item $\rho\lambda\rho\leq\rho$.
\end{enumerate}
\end{lem}
\begin{proof}
Right adjoint extensions have been shown to be left regular via the isomorphism of left $B$-modules $f:\,_BA\iso\,_BB$, $a\mapsto e\lambda(a)$. This induces the isomorphism
\begin{align}
\vartheta:\ A\oB A&\iso A\oB B\cong A \label{vartheta}\\
a'\oB a&\mapsto a'\oB e\lambda(a)\mapsto a'\rho(e\lambda(a)) \notag
\end{align}
of $A$-$B$ bimodules provided we set the right $B$-action $a\cdot b:=a\rho\lambda\rho(b)$ on the image copy of $A$. This can be denoted by $_A(A\oB A)_B\rarr{\vartheta}\, _AA_{\lambda\rho(B)}$.
Since left $A$-module maps on $A$ are right multiplications, direct summand diagrams
\[
\begin{CD}
_AA_{\lambda\rho(B)}@>\pi_i>>_AA_B@>\iota_i>>_AA_{\lambda\rho(B)}
\end{CD}
\]
in $_A\M_B$ are in bijection with direct summand diagrams
\begin{equation} \label{p q}
\begin{CD}
\rho\lambda\rho@>p_i>>\rho@>q_i>>\rho\lambda\rho
\end{CD}
\end{equation}
in $\Ring(B,A)$ via $q_i=\pi_i(1_A)$, $p_i=\iota_i(1_A)$ on the one hand and
$\pi_i(a)=aq_i$, $\iota_i(a)=ap_i$ on the other hand. This proves the equivalence.
\end{proof}

\begin{lem} \label{lem: Ha}
For a morphism $\rho:B\to A$ with a left adjoint $\lambda$ assume $\rho\lambda\rho\leq\rho$. Then there is a left bialgebroid $\Ha$ over $R=\End\rho$
defined by
\begin{align*}
\Ha&=\End(\rho\lambda)\quad\text{as a ring}\\
s_\Ha&:R\to\Ha,\quad r\mapsto r\\
t_\Ha&:R^\op\to\Ha,\quad r\mapsto \rho(r^\lef)\\
\cop_\Ha&:\Ha\to\Ha\oR\Ha,\quad h\mapsto \sum_i\rho(e)h\rho(e\lambda(q_im))\oR p_i\rho\lambda(m)\\
\eps_\Ha&:\Ha\to R,\quad h\mapsto\rho(e)hm
\end{align*}
which is finite projective as a left $R$-module.
\end{lem}
\begin{proof}
Since the canonical $R$-$R$-bimodule structure of a left bialgebroid is defined by
$r\cdot h\cdot r'=s_\Ha(r)t_\Ha(r')h$, finite projectivity of $_R\Ha$ follows from the
adjunction isomorphism $\Hom(\rho\lambda,\rho\lambda)\cong\Hom(\rho\lambda\rho,\rho)$,
$h\mapsto \rho(e)h$ and from Lemma \ref{Morita}.

The ring extension $\rho$ being right depth 2 by Lemma \ref{lem: D2} the proof of
the left bialgebroid axioms for $\Ha$ can be considered standard. They follow e.g. from the construction of \cite{Ka-Sz} with the important additional observation made in \cite{Ka: Gal bgd} that a one-sided depth two condition suffices. (We note that bialgebroids
from one-sided depth two arrows in $\Ab$-bicategories has already been constructed in
\cite[Subsection 3.3]{Sz: Siena} albeit with the further assumption $\iota\leq\lambda\rho$.)
\end{proof}

Returning to field algebras we get the following specialization of Proposition
\ref{pro: bgd cosym}.
\begin{pro}
Given an adjunction $\lambda\dashv\rho$ in $\Ring$ 
let $\C\subset\Imp(\rho)$ be a full monoidal subcategory containing the object $\lambda\rho$.
Then 
\begin{enumerate}
\item $F=\Hom(\rho,\rho\under):\C\to\Ab$ is a representable essentially strong monoidal functor
with $\eta_F$ an isomorphism.
\item The field algebra extension $\rho_F:B\to\Upsilon\amo{\C}F$ is isomorphic in $\Bslice$ to the original extension $\rho$ via $\eps_\rho$.
\item $H=F^\lef\amo{\C}F$ of Proposition \ref{pro: bgd cosym} is the opposite bialgebroid
\cite{Ka-Sz} of the left bialgebroid $\Ha$ defined in Lemma \ref{lem: Ha}
\item The field algebra extension $\rho_F$ is an $H$-Galois extension over the right bialgebroid $H=F^\lef\amo{\C}F$ of Proposition \ref{pro: bgd cosym} with coinvariant subalgebra equal
to $\rho(B)$.
\item $\rho$ is an $\Ha$-Galois extension over the left bialgebroid $\Ha$ with coinvariant
subalgebra equal to $\rho(B)$.
\end{enumerate}
\end{pro}
\begin{proof}
$(1)$ Since $\gamma:=\lambda\rho\in\C\subset\Imp(\rho)$, both Proposition \ref{the good F}
and Lemma \ref{lem: rep}(2) apply.

$(2)$ The composite  $\Upsilon\amo{\C}F\iso\Upsilon\amo{\C}\Hom(\gamma,\under)\iso\Upsilon(\lambda\rho)
=B_\gamma$ sending $b\amo{\alpha}x$ to $be\lambda(x)$ is an isomorphism of $B$-rings
which, when composed with $\kappa:B_\gamma\to A$, yields $b\amo{\alpha}x\mapsto
\rho(be\lambda(x))m=\rho(b)x=\eps_\rho(b\amo{\alpha}x)$.

$(3)$ Since $\lambda\rho\in\Imp(\rho)$, the depth 2 condition $\rho\lambda\rho\leq\rho$ holds and therefore $\Ha$ is a left $R$-bialgebroid by Lemma \ref{lem: Ha}. Then the isomorphism  
\begin{align*}
F^\lef\amo{\C}F&\iso F^\lef\amo{\C}\Hom(\gamma,\under)\iso F^\lef(\gamma)=
\Hom(\rho\gamma,\rho)\iso\End(\rho\lambda)=\Ha\\
f\amo{\alpha}x&\mapsto f\amo{\alpha}e\lambda(x)\mapsto f\rho(e\lambda(x))\mapsto
f\rho(e\lambda(x))\rho\lambda(m)=f\rho(e\lambda(xm))
\end{align*}
is the required bialgebroid map from $H$ to $\Ha^\op$.

$(4)$ will follow from $(5)$ because of $(2)$ and $(3)$.

$(5)$ Since coactions do not see the ring structure of the bialgebroid, we can apply the isomorphisms described in $(2)$ and $(3)$ to the $H$-coaction (\ref{delta}) to get the $\Ha$-coaction
\[
\Delta_A:A\to A\oR \Ha,\quad a\mapsto a\nullB\oR a\oneB
=\sum_i\rho(e\lambda(a))q_i\oR p_i\rho\lambda(m).
\]
It is an $R$-$R$-bimodule map, i.e.,
\[
\Delta_A(r'ar)=a\nullB\oR\, t_\Ha(r)a\oneB t_\Ha(r')\,\qquad r,r'\in R,\ a\in A
\]
holds where the $R$-actions the $\oR$ sign is referring to are simply right multiplication on $A$ by
$r\in R\subset A$ and left multiplication on $\Ha$ with $r=s_\Ha(r)$.
The coaction factors through the multiplicative subbimodule $A\ex{R}\Ha\hookrightarrow A\oR\Ha$, the Takeuchi product, because of the centrality  property
\[
ra\nullB\oR a\oneB=a\nullB\oR a\oneB r\,\qquad r\in R,\ a\in A
\]
where the $r$ on the RHS multiplies according to multiplication in $\Ha$ since $R\subset H$.
The comodule algebra properties then read as
\begin{align*}
a\nullB a'\nullB\oR a'\oneB a\oneB&=(aa')\nullB\oR (aa')\oneB\\
\Delta_A(1_A)&=1_A\oR 1_\Ha\,.
\end{align*}
Having been constructed the $\Ha$-comodule algebra $\bra A,\Delta_A\ket$ the next task is to construct an inverse of the to-be-Galois map $\Gamma:A\oB A\to A\oR\Ha$ prior to knowing that
$B$ exhausts  the coinvariants. Composing $\Gamma$ with the inverse $a\mapsto a\oB m$ of (\ref{vartheta})
we get
\[
\Gamma\ci\vartheta^{-1}(a)=\Gamma(a\oB m)=\sum_i aq_i\oR p_i\rho\lambda(m)
\]
which in turn has inverse
\[
(\Gamma\ci\vartheta^{-1})^{-1}(a\oR h)=a\rho(e)h
\]
since $\rho(e)hq_i\in R$ for $h\in\Ha$. But then
\[
\vartheta^{-1}\ci(\Gamma\ci\vartheta^{-1})^{-1}(a\oR h)=\vartheta^{-1}(a\rho(e)h)= a\rho(e)h\oB m
\]
provides the inverse of $\Gamma$.

In order to find the coinvariant subalgebra we compute
\begin{align*}
(\Gamma\ci\vartheta^{-1})^{-1}\ci\Delta_A(a)&=a\nullB\rho(e)a\oneB=\rho(e\lambda(a))\\
(\Gamma\ci\vartheta^{-1})^{-1}(a\oR 1_\Ha)&=a\rho(e)\,.
\end{align*}
Therefore if $a$ is a coinvariant then $a\rho(e)=\rho(e\lambda(a))$ implying that $a\rho(e)\in\rho(B)$.
But then $a=a\rho(e)\rho\lambda(m)\in\rho(B)$. The opposite implication $a\in\rho(B)$ implying
$\rho(e\lambda(a))=a\rho(e)$ is obvious.
\end{proof}

The next Corollary which parallels \cite[Theorem 2.1]{Ka: Gal bgd} and \cite[Theorem 3.6]{Ba-Sz},
characterizes the extensions we are studying here independently of the field algebra construction.
\begin{cor}
For a ring homomorphism $\rho:B\to A$ the following conditions are equivalent:
\begin{enumerate}
\item $\rho$ is right $H$-Galois for some left finite projective right bialgebroid $H$ over $R=\End\rho$
and $_BA\cong\,_BB$. (Left regular Galois extension)
\item $\rho$ is right adjoint in $\Ring$ and $\rho\lambda\rho\leq\rho$ for some, and then any, left adjoint $\lambda$ of $\rho$.
\end{enumerate}
\end{cor}

Finally we remark that, because of finite projectivity of $_R\Ha$, the category of right $\Ha$-comodules
is (monoidally) isomorphic to the category of right $\G$-modules where $\G\cong\Hom_{R-}(\Ha,R)$
is the dual right bialgebroid of $\Ha$. It is defined by
\begin{align*}
\G&=\End(\lambda\rho)\quad\text{as a ring}\\
s_\G&:R\to\G,\quad r\mapsto \lambda(r)\\
t_\G&:R^\op\to\G,\quad r\mapsto r^\lef\equiv e\lambda(rm)\\
\cop_\G&:\G\to\G\oR\G,\quad g\mapsto\sum_ie\lambda(q_i)\oR e\lambda(p_i\,\rho\lambda(m)\rho(g)m)\\
\eps_\G&:\G\to R,\quad g\mapsto \rho(eg)m
\end{align*}
and the duality is given by the $R$-valued bilinear form 
\[
\bra h, g\ket=\rho(e)h\rho(g)m\,,\qquad h\in\Ha,\ g\in\G\,.
\]
The above Proposition implies that the field algebra is a right $\G$-module algebra with action $a\ract g=\rho(e\lambda(a)g)m$, with invariant subalgebra $\rho(B)\cong B$ and this action is Galois in the sense of the
smash product $\G\mash A$ being isomorphic to $\End(_BA)$ via the map
$g\mash a\mapsto \{a'\mapsto(a'\ract g)a\}$.

\subsection{Generalized fusion categories}
Recall Subsection \ref{ss: preorder} that a monoid in the monoidal preorder $\Po$-$\Ring(B,B)$ is an object
$\sigma$ such that $\sigma^2\leq\sigma$ and $\iota\leq\sigma$. Notice that this notion is a property of the object
and not a structure. Then we make the following elementary observations:

\begin{lem}
For an object $\sigma\in\Ring(B,B)$ let $\C_\sigma$ denote the full subcategory of $\Ring(B,B)$ the objects $\alpha$ of which satisfy $\alpha\leq\sigma$.
\begin{enumerate}
\item $\C_\sigma$ is equivalent to a full subcategory of the category of finitely generated projective right $S$-modules where $S=\End \sigma$.
\item $\C_\sigma$ is a monoidal subcategory of $\Ring(B,B)$ if and only if $\sigma$ is a monoid in the preorder $\Po(B,B)$.
\end{enumerate}
\end{lem}
\begin{exa} Since the construction of the preorder $\Po$ of Subsection \ref{ss: preorder} can be applied to
any 2-category (or bicategory) enriched over $\Ab$, in any fusion category $\C$ \cite{fusion} the direct sum
$\sigma=\oplus_i\sigma_i$ of representants of the simple objects of $\C$ is a monoid in the associated monoidal
preorder. 
\end{exa}

Let $\sigma$ be a monoid in the preorder and fix direct summand diagrams
\[
\sigma^2\rarr{p_i}\sigma\rarr{q_i}\sigma^2\,,\qquad \iota\rarr{u_i}\sigma\rarr{e_i}\iota\,.
\]
Let $F$ be any additive monoidal functor $\C_\sigma\to \Ab$ and let $\Upsilon\amalgo{\C_\sigma}F$ be the ring
associated to it. 
Every element of this ring is a sum of terms
\[
b\amo{\alpha}x=\sum_ib\amo{\alpha}F(t_is_i)x=\sum_i bt_i\amo{\sigma}F(s_i)x
\]
with $b\in B$, $x\in F(\alpha)$
where $\alpha\rarr{s_i}\sigma\rarr{t_i}\alpha$ is a direct
summand diagram which exists by the assumption $\alpha\leq\sigma$. Therefore
\begin{equation} \label{identify}
\Upsilon\am{\C_\sigma}F\iso B\am{S} F(\sigma),\quad b\amo{\alpha}x\mapsto\sum_i bt_i\am{S}F(s_i)x
\end{equation}
is an isomorphism inducing a ring structure on $A:=B\am{S}F(\sigma)$
with multiplication rule and unit
\begin{align*}
(b\am{S}x)(b'\am{S}x')&=\sum_i b\sigma(b')q_i\am{S}F(p_i)F_{\sigma,\,\sigma}(x\o x')\,,\\
1_A&=\sum_i e_i\am{S}F(u_i)F_0(1)\,,
\end{align*}
respectively. Of course we are interested in functors of the form $\Hom(\rho,\rho\under)$,
therefore we set 
\begin{align*}
&F:\C_\sigma\to\Ab\,,\quad F(\alpha)=\Hom(\sigma,\sigma\alpha)\\
&F_{\alpha,\beta}(t\o s)=ts\,,\quad F_0(1)=1_S\,.
\end{align*}
The corresponding ring $A=B\am{S} F(\sigma)$ has multiplication rule and unit
\begin{align*}
(b\am{S}x)(b'\am{S}x')&=\sum_i b\sigma(b')q_i\am{S}\sigma(p_i)xx'\\
1_A&=\sum_i e_i\am{S}\sigma(u_i)\,.
\end{align*}
Note that $F(\sigma)=\Hom(\sigma,\sigma^2)$, as every $F(\alpha)$, too, has an $S$-$S$-bimodule structure.
But the left $S$-module $_SF(\sigma)$ tensored in $B\am{S}F(\sigma)$ is of a third type: $s\cdot x=\sigma(s)x$.

The field algebra extension $\E(F)$ corresponds, under (\ref{identify}), to the morphism $\rho:B\to A$, $\rho(b)=\sum_i be_i\am{S}\sigma(u_i)$ and the associated bimodule structure on $A$ is
\begin{align} \label{_BA_B}
\rho(b_1)(b\am{S}x)\rho(b_2)&=(b_1b\am{S}x)\rho_F(b_2)=\sum_{i,j}b_1b\sigma(b_2e_i)q_j\am{S}
\sigma(p_j)x\sigma(u_i)= \notag\\
&=b_1b\sigma(b_2)\am{S}x
\end{align}
where we used that $p_j\sigma(u_i)\in S$. 

\begin{lem} \label{K}
For $\sigma$ a monoid in $\Po(B,B)$ and $F=\Hom(\sigma,\sigma\under):\C_\sigma\to\Ab$ the $F$-component
of $\eta$, the unit of the adjunction $\E\dashv\F$, is a split monomorphism. The splitting map is the restriction
to the hom-groups of a ring homomorphism $\mathbb{K}:\Upsilon\amo{\C_\sigma}F\to B$
satisfying $\mathbb{K}\rho_F=\sigma$.
\end{lem}
\begin{proof}
Upon the identification (\ref{identify}) $\rho_F$ corresponds to $\rho$ and (the $\alpha$-component of) $\eta_F$
to the composite
\begin{align*}
F(\alpha)=\Hom(\sigma,\sigma\alpha)&\rarr{\eta}\Hom(\rho_F,\rho_F\alpha)\iso\Hom(\rho,\rho\alpha)\\
x&\mapsto 1_B\amo{\alpha}x\mapsto\sum_i t_i\amo{\sigma}\sigma(s_i)x
\end{align*}
where $t_i,s_i\in B$ are chosen by $\alpha\leq\sigma$. This map is clearly split by $b\am{S}x\mapsto\sigma(b)x$.
Then the map
\begin{equation} \label{def K}
K:A=B\am{S}F(\sigma)\to B,\qquad b\am{S}x\mapsto\sigma(b)x
\end{equation}
is a ring homomorphism because
\begin{align*}
K((b\am{S}x)(b'\am{S}x'))&=\sum_i\sigma(b\sigma(b')q_i)\sigma(p_i)xx'=\sigma(b)x\sigma(b')x'=
K(b\am{S}x)K(b'\am{S}x')\\
K(1_A)&=\sum_i\sigma(q_i)\sigma(p_i)=1_B\,.
\end{align*}
It satisfies $K(\rho(b))=\sum_i\sigma(bq_i)\sigma(p_i)=\sigma(b)$ for $b\in B$. The required morphism $\mathbb{K}$ is then $K$ composed with (\ref{identify}).
\end{proof}

\begin{pro} \label{pro: fusion}
Let $\sigma$ be a monoid in $\Po(B,B)$ possessing a left dual $\sigma^\lef\in\C_\sigma$ and let $S=\End\sigma$.
Then
\begin{enumerate}
\item $\C_\sigma=\Imp(\sigma)$
\item $F=\Hom(\sigma,\sigma\under):\C_\sigma\to\Ab$ is a representable essentially strong monoidal functor
with its strong factor $\hat F:\C_\sigma\to\,_S\M_S$ mapping each object $\alpha$ to a bimodule $\hat F(\alpha)$ that is finite projective as right $S$-module. 
\item The field algebra extension $\rho_F:B\to\Upsilon\am{\C_\sigma}F$ has a left adjoint $\lambda_F$
satisfying $\rho_F\lambda_F\rho_F\leq\rho_F$ and $\iota\leq\lambda_F\rho_F$, hence also
$\rho_F\lambda_F\rho_F\sim\rho_F$.
\item The rings $B$ and $\Upsilon\am{\C_\sigma}F$ are isomorphic.
\item The $\eta_F:F\to\F\E(F)$ is an isomorphism.
\end{enumerate}
\end{pro}

\begin{proof}
$(1)$ For any object $\alpha\in\Ring(B,B)$ the relation $\alpha\leq\sigma$ implies $\sigma\alpha\leq\sigma^2\leq\sigma$, hence $\C_\sigma\subset\Imp(\sigma)$, and the relation $\sigma\alpha\leq\sigma$ implies $\alpha=\iota\alpha\leq\sigma\alpha\leq\sigma$, hence $\Imp(\sigma)\subset\C_\sigma$.

$(2)$ This follows from Proposition \ref{the good F} and from $(1)$.

$(3)$ By the correspondence (\ref{identify}) it suffices to show that $\rho:B\to A=B\am{S}F(\sigma)$
has a left dual $\lambda$ such that 
$\rho\lambda\rho\leq\rho$ and $\iota\leq\lambda\rho$. At first we show that $K:A\to B$ has inverse
\[
K^{-1}(b)=\sum_i\hat e\sigma^\lef(b)\tilde q_i\am{S}\sigma(\tilde p_i)\hat m
\]
where $\hat e,\hat m\in B$ are the counit and unit of $\sigma^\lef\dashv\sigma$ and
\[
\sigma^\lef\sigma\rarr{\tilde p_i}\sigma\rarr{\tilde q_i}\sigma^\lef\sigma
\]
is a direct summand diagram. Indeed,
\begin{align*}
K^{-1}(K(b\am{S}x))&=\sum_i\hat e\sigma^\lef(\sigma(b)x)\tilde q_i\am{S}\sigma(\tilde p_i)\hat m=\\
&\sum_i b\,\underset{\in S}{\underbrace{\hat e\sigma^\lef(x)\tilde q_i}}\am{S}\sigma(\tilde p_i)\hat m
=b\am{S}\sigma(\hat e)\hat mx=\\
&=b\am{S}x\\
K(K^{-1}(b))&=\sigma(\hat e\sigma^\lef(b))\hat m=b\,.
\end{align*}
Then defining $\lambda:=\sigma^\lef K:A\to B$ we have $\lambda\dashv\rho$ with counit $e=\hat e:\sigma^\lef\sigma=\lambda\rho\to\iota_B$ and unit $m=K^{-1}(\hat m):\iota_A=K^{-1}K\to K^{-1}\sigma\sigma^\lef K=\rho\lambda$. The depth two relation then follows easily since $\rho\lambda\rho=K^{-1}\sigma\sigma^\lef\sigma\leq K^{-1}\sigma^3\leq K^{-1}\sigma=\rho$.
The relation $\iota\leq\lambda\rho$ follows from $\lambda\rho=\sigma^\lef\sigma$ and from $\iota\leq\sigma$
and $\iota=\iota^\lef\leq\sigma^\lef$. 

$(4)$ The isomorphism is given by $K$ and (\ref{identify}).

$(5)$ We have seen already in Lemma \ref{K} that $\eta_F$ splits by the restriction of $K$. Since now $K$ is an isomorphism, it suffices to show that $K(\Hom(\rho,\rho\alpha))\subset\Hom(\sigma,\sigma\alpha)$ for all $\alpha$.
By (\ref{_BA_B})
\[
\sum_j b_j\am{S}x_j\in\Hom(\rho,\rho\alpha)\quad\Leftrightarrow\quad
\sum_jb_j\sigma(b)\am{S}x_j=\sum_j\alpha(b)b_j\am{S}x_j
\]
for all $b\in B$. Hence applying (\ref{def K}) to such elements we obtain for all $b\in B$ 
\[
\sum_j\sigma(b_j)x_j\,\sigma(b)=\sum_j\sigma(b_j\sigma(b))x_j=\sigma\alpha(b)\,\sum_j\sigma(b_j)x_j
\]
which proves the claim.
\end{proof}

Since the $\rho_F$ obtained in the Proposition satisfy the requirements the $\rho$ had in Proposition \ref{pro: bgd cosym}, we immediately obtain most of the statements in the
\begin{cor}
For $\sigma$ a monoid in $\Po(B,B)$ possessing a left dual $\sigma^\lef\in\C_\sigma$ let $S=\End\sigma$
and $F$ be the monoidal functor $\Hom(\sigma,\sigma\under):\C_\sigma\to\Ab$. 
Then the field algebra extension $\rho_F:B\to\Upsilon\am{\C_\sigma}F$ is a split $B$-ring and right $\Ha$-Galois
over the left bialgebroid $\Ha=\End(\sigma\sigma^\lef)$ which in turn is a split $S$-ring via the source map
$s_\Ha:S\to \Ha$.
\end{cor}
\begin{proof}
We need to prove splitness of $\rho:B \to A$ and $s_\Ha:R\to\Ha$. (Note that $S\cong R$ by taking the
$\alpha=\iota$ component of $\eta_F$ in Proposition \ref{pro: fusion} (5).)
Splitness of an $N$-ring $N\to M$ means the existence of a bimodule map $E:\,_NM_N\to\,_NN_N$
such that $E(1_M)=1_N$. Let $\iota\rarr{v_k}\lambda\rho\rarr{w_k}\iota$ be a direct sum diagram which exists
by Proposition \ref{pro: fusion} (3). Then $E:A\to B$ defined by $E(a)=\sum_k w_k\lambda(a)v_k$ is such a unit
preserving bimodule map for $\rho$. For $s_\Ha:R\to\Ha$, $r\mapsto r$ a unit preserving bimodule map can be given by $E'(h)=\sum_k\rho(w_k)h\rho(v_k)$.
\end{proof}

We note that the condition $\sigma^\lef\leq\sigma$ in Proposition is equivalent to $\sigma$
having a left adjoint which is implementable in $\sigma$. Furthermore, such left adjoints are also monoids in
$\Po(B,B)$.
However, the relation $\sigma^\lef\leq\sigma$ does not imply $\sigma\leq\sigma^\lef$.
If we assume both, i.e., $\sigma^\lef\sim\sigma$, then $S$ becomes Morita equivalent to
$S^\op\cong\End(\sigma^\lef)$.
If we make the stronger assumption that $\sigma^\lef$ is also a right adjoint of $\sigma$ then by general arguments
\cite{Bo-Sz} we know to obtain Frobenius Hopf algebroids for the symmetry of the field algebra extension.


\begin{thebibliography}{XX} 
\begin{small}

\bibitem{Ro} J.E. Roberts, "The reconstruction problem", unpublished manuscript

\bibitem{Ha-Mu} H. Halvorson, M. M\"uger, "Algebraic quantum field theory",
\textit{Handbook of the Philosophy of Science, Philosophy of Physics}, J. Butterfield, J. Earman eds,
Elsevier, 2007. \texttt{ArXiV:math-ph/0602036}

\bibitem{Do-Ro} S. Doplicher, J.E. Roberts, "Endomorphisms of C$^*$-algebras, cross products and duality for compact groups", \textit{Ann. Math.}, \textbf{130} (1989) pp. 75-119.

\bibitem{Jo-St} A. Joyal, R. Street, "An Introduction to Tannaka Duality and Quantum Groups",
Category Theory (Como, 1990), \textit{Lecture Notes in Math.}, \textbf{1488}, pp. 413-492,
A. Carboni, M.C. Pedicchio, G. Rosolini eds, Springer, 1991.

\bibitem{Fr} P. Freyd: \textit{Abelian Categories}, Harper\&Row, 1964.

\bibitem{Sz: Brussels} K. Szlach\'anyi, "Adjointable monoidal functors and quantum groupoids",
in \textit{Hopf algebras in Noncommutative Geometry and Physics}, pp. 291-307,
S. Caenepeel, F. Van Oystaeyen eds, Marcel Dekker, 2005.

\bibitem{Phung} Ph\`ung H{\`o} H\'ai, "Tannaka-Krein duality for Hopf algebroids,
\textit{Israel J. Math.}, \textbf{167} (2008) pp. 193-226. 

\bibitem{CWM} Mac Lane, S.: \textit{Categories for the Working
Mathematician}, 2nd edition, GTM 5, Springer-Verlag New-York Inc., 1998.

\bibitem{Borceux} F. Borceux, \textit{Handbook of Categorical Algebra},
Cambridge University Press, 1994.

\bibitem{TTT} M. Barr, C. Wells, \textit{Toposes, Triples and
Theories},\newline \texttt{http://www.cwru.edu/artsci/math/wells/pub/ttt.html}

\bibitem{Ka-Sz} L. Kadison, K. Szlach\'anyi, "Bialgebroid actions on depth two extensions and duality",
\textit{Advances in Mathematics}, \textbf{179} (2003) pp. 75-121.

\bibitem{Sz: Siena} K. Szlach\'anyi, "Finite quantum groupoids and inclusions of finite type",
\textit{Fields Inst. Comm.}, \textbf{30} (2001) pp. 393-407.

\bibitem{Ka: Gal bgd}L. Kadison, "Galois theory for bialgebroids, depth two and normal Hopf subalgebras",
\textit{Ann. Univ. Ferrara - Sez. VII - Sc. Mat.}, \textbf{51} (2005) pp. 209-231.

\bibitem{Ba-Sz} I. B\'alint, K. Szlach\'anyi, "Finitary Galois theory over noncommutative bases",
\textit{J. Algebra}, \textbf{296} (2006) pp. 520-560.

\bibitem{fusion} P. Etingof, D. Nikshych, V. Ostrik, "On fusion categories",
\textit{Ann. of Math.}, \textbf{162}  (2005) pp. 581-642.

\bibitem{Bo-Sz} G. B\"ohm, K. Szlach\'anyi, "Hopf Algebroid Symmetry of Abstract Frobenius Extensions of Depth 2", \textit{Commun. Algebra}, \textbf{32} (2004) pp. 4433-4464.

\end{small}
\end{thebibliography}
\end{document}